\documentclass[a4paper, 11pt]{article}
\usepackage[latin1]{inputenc}    
\usepackage[T1]{fontenc}
\usepackage{lmodern}
\usepackage{graphicx}
\usepackage{subfigure}
\usepackage{amsmath, amsfonts, amssymb, amsthm} 
\usepackage{bbold}
\usepackage{framed}
\usepackage{color}
\usepackage{tikz}
\usepackage[hidelinks]{hyperref} 
\usepackage{xcolor}
\usepackage{enumerate} 

\hypersetup{
    colorlinks,
    linkcolor={red!80!black},
    citecolor={blue!80!black},
    urlcolor={blue!80!black}
}

\usepackage{ulem}
\numberwithin{equation}{section} 

\newtheorem{define}{Definition}[section]
\newtheorem{proposition}{Proposition}[section]
\newtheorem{theorem}{Theorem}
\newtheorem{corollary}{Corollary}[section]
\newtheorem{remark}{Remark}[section]
\newtheorem{lemma}{Lemma}[section]
\newtheorem{assume}{Assumption}

\definecolor{myred}{cmyk}{0, 1, 1, 0}

\newcommand{\E}{\mathbb{E}}
\newcommand{\Proba}{\mathbb{P}}

\newcommand*\myref[1]{{\normalfont (\ref{#1})}}
\newcommand*\mycite[1]{{\normalfont \cite{#1}}}

\title{Weak and strong mean-field limits for stochastic Cucker-Smale particle systems}
\date{\vspace{-5ex}}
\author{Angelo Rosello {\vspace{2mm}} \\
\textit{
Univ Rennes, CNRS, IRMAR - UMR 6625, F-35000 Rennes, France }
}

\begin{document}

\maketitle

\begin{abstract}
We consider a particle system with a mean-field-type interaction perturbed by some common and individual noises. When the interacting kernels are sublinear and only locally Lipschitz-continuous, relying on arguments based on the tightness of random measures in Wasserstein spaces, we are able to construct a weak solution of the corresponding limiting SPDE. In a setup where the diffusion coefficient on the environmental noise is bounded, this weak convergence can be turned into a strong $L^p(\Omega)$ convergence and the propagation of chaos for the particle system can be established. The systems considered include perturbations of the Cucker-Smale model for collective motion.
\vspace{3mm}

\textbf{Keywords:} stochastic particle systems, mean-field limit, propagation of chaos, stochastic partial differential equations, Cucker-Smale model, collective motion.
\end{abstract}

\vspace{10mm}

\tableofcontents


\clearpage

\begin{section}{Introduction} \label{1}

\begin{subsection}{Overview of the model.}

Flocking, or swarming, is a phenomenon consistently observed in nature where individuals from a population (birds, fish, insects, bacterias...) tend to naturally align their trajectories without the need of a leadership. One of the most commonly studied model which intends to describe this kind of behavior is the Cucker-Smale model, introduced in \cite{cs1} and \cite{cs2}.

 In this model, each individual interacts with the group in a mean-field-like manner: denoting by $X^{i,N}$, $V^{i,N} \in \mathbb{R}^d$ the position and velocity of the $i$-th individual, the behavior of the system can be written as
\begin{equation}
\left\{
\begin{array}{l c l}
\displaystyle{ \frac{d}{dt} } X^{i,N}_t & = & V^{i,N}_t \\
\displaystyle{ \frac{d}{dt} } V^{i,N}_t & = & \displaystyle{ \frac{1}{N} \sum_{j=1}^N \psi(X^{i,N}_t - X^{j,N}_t)(V^{j,N}_t - V^{i,N}_t) }
\end{array}
\right.
\label{cs}
\end{equation}
where the weight function $\psi : \mathbb{R} \to \mathbb{R}^+$ is even and bounded, typically of the form
\begin{align*}
\psi(x-y) = \frac{\lambda}{(1 + |x-y|^2)^\gamma}, \; \; \; \lambda, \gamma > 0.
\end{align*}
In order to take into account unpredictable phenomena of different natures, it is rather natural to perturb this deterministic model with some noise. 
In \cite{oldandnew}, where the flocking phenomenon (alignment of speeds, distance between the individuals bounded over time) is studied in a variety of different stochastic Cucker-Smale models, three different kinds of perturbations are identified. 

The first one considers the degree of freedom of each individual by adding some independent noise, dragged by a brownian motion $B^i$, to each of them:
\begin{align}
dV^{i,N}_t  =  \frac{1}{N} \sum_{j=1}^N \psi(X^{i,N}_t - X^{j,N}_t)(V^{j,N}_t - V^{i,N}_t) dt + \sigma(X^{i,N}_t, V^{i,N}_t) \circ dB^i_t.
\label{pert1}
\end{align}
This setting typically appears in the propagation of chaos framework. The flocking behavior for \myref{pert1} has been studied in \cite{ha_lee}. The mean-field limit as $N$ goes to infinity is considered in \cite{bolley}, in the case of a constant diffusion coefficient $\sigma(x,v) = \sqrt{D} Id$, and more recently in \cite{choi-salem2} for $\sigma(x,v) = {\cal R}(v)$ a "truncation function" of the speed.
Note that, when presenting new models, we insist on introducing noise in Stratonovich form, since it is the most physically relevant form.

Another kind of perturbation might emerge from the environment in which the individuals evolve. In this case, we add some common noise dragged by a Wiener process $dW = \sum_k \sigma_k dW^k$:
\begin{align}
dV^{i,N}_t  =  \frac{1}{N} \sum_{j=1}^N \psi(X^{i,N}_t - X^{j,N}_t)(V^{j,N}_t - V^{i,N}_t) dt + \sum_k \sigma_k(X^{i,N}_t, V^{i,N}_t) \circ dW^k_t.
\label{pert2}
\end{align}
A version of \myref{pert2}, with a diffusion coefficient of the form $\sigma(x,v) = D(v-v_e)$ for some constant $v_e \in \mathbb{R}^d$, is studied in \cite{ahn_mi_ha}.

Lastly, one may consider that the weight function $\psi$ modeling the interaction between individuals is perturbed into $\tilde \psi = \psi + d\xi$, where $\xi$ is some space-dependent Wiener process given by $d\xi = \sum_k \phi_k d\beta^k$, leading to
\begin{align}
dV^{i,N}_t  = &  \frac{1}{N} \sum_{j=1}^N \psi(X^{i,N}_t - X^{j,N}_t)(V^{j,N}_t - V^{i,N}_t) dt  \nonumber
\\
& \hspace{10mm} + \frac{1}{N} \sum_{j=1}^N \sum_k \phi_k(X^{i,N}_t - X^{j,N}_t) (V_t^{j,N} - V_t^{i,N}) \circ d\beta^k_t.
\label{pert3}
\end{align}
The mean-field limit and flocking for \myref{pert3} is looked upon in \cite{choi-salem} and more recently in \cite{jung-ha} in the particular case where the perturbation $\xi$ does not depend on $x$:  $d\xi = \sqrt{2 \sigma} d\beta_t$.
\vspace{3mm}

In this paper, we focus on the mean-field limit of these particle systems. Namely, we intend to extend the results mentioned above by studying the behavior of the empirical measure 
$$\mu^N = \frac{1}{N} \sum_{i=1}^N \delta_{(X^{i,N}, V^{i,N})} $$ 
as $N$ goes to infinity, for general stochastic Cucker-Smale model of the form \myref{pert1}, \myref{pert2} or \myref{pert3} (or combinations of these models). Let us keep the notions of convergence a little vague for a moment, in order to give a quick overview of the results to come: we will show for instance that, for \myref{pert3}, under the assumptions
\begin{align}
\sum_k \| \phi_k \|^2_\infty < \infty, \; \; \; \sum_k \| \phi_k \|^2_{\text{lip}} < \infty,
\label{ass_phi}
\end{align}
where $\| \phi \|_{lip} = \sup_{x \neq y} \frac{ |\phi(x)-\phi(y)|}{|x-y|}$, provided that $\mu_0^N \to \mu_0$, the (random) empirical measure $\mu^N_t$ converges in law, up to a subsequence, to some $\mu_t$ which is a weak solution of the expected limiting stochastic PDE
\begin{align}
& d\mu_t + v \cdot \nabla_x \mu_t dt + \nabla_v \cdot \left( F[\mu_t] \mu_t \right) dt + \sum_k \nabla_v \cdot \left( F_k[\mu_t] \mu_t \right) \circ d\beta_t^k = 0,
\label{cs_spde}
\end{align}
with
\begin{align*}
& F[\mu](x,v) = \int \psi(x-y)(w-v) d\mu(y,w), \; \; \; F_k[\mu](x,v) = \int \phi_k(x-y)(w-v) d\mu(y,w) .
\end{align*}
It is of some interest to note here that the noise added in Stratonovich form in \myref{pert3} directly translates into the expected conservative form \myref{cs_spde} for the limiting equation, which emphasizes the physical relevance of Stratonovich's integration over Itô's.

Regarding the flocking phenomenon, the method developed in \mycite{choi-salem} could in fact be easily extended to the model \myref{cs_spde}. Given a solution $\mu=(\mu_t)_{t \ge 0}$ of \myref{cs_spde}, the average velocity $\bar v_t~=~\int v d\mu_t$ is conserved over time. Assuming that 
$$\psi_m := \min_x \psi(x) > 0, \hspace{15mm} \sum_k \| \phi_k \|^2_\infty < \infty$$
 and denoting
$$E_t = \int_{\mathbb{R}^{2d}} |v - \bar v_t|^2 d\mu_t(x,v)$$
calculations easily lead to
$$
\frac{d}{dt} \E[E_t] \le -2 \left(\psi_m - 4 \sum_k \| \phi_k \|^2_\infty \right) \E[ E_t ].
$$
Therefore, under the condition $\psi_m > 4 \sum_k \| \phi_k \|^2_\infty$, the model \myref{cs_spde} exhibits a flocking behavior in the sense that $\E [ E_t ] \to 0$ exponentially fast as $t$ goes to infinity.
\vspace{3mm}

Under the same assumptions \myref{ass_phi}, the "strong" mean-field convergence $\mu_t^N \to \mu_t$ (see Theorem \ref{chap2-thm2} below for details) is obtained for the whole sequence if we consider "truncated velocities" in the perturbative term, that is a model given by
\begin{align}
dV^{i,N}_t  = &  \frac{1}{N} \sum_{j=1}^N \psi(X^{i,N}_t - X^{j,N}_t)(V^{j,N}_t - V^{i,N}_t) dt  \nonumber
\\
& \hspace{10mm} + \frac{1}{N} \sum_{j=1}^N \sum_k \phi_k(X^{i,N}_t - X^{j,N}_t) {\cal R}(V_t^{j,N} - V_t^{i,N}) \circ d\beta^k_t
\label{pert4}
\end{align}
where ${\cal R} : \mathbb{R}^d \to \mathbb{R}^d$ is smooth and compactly-supported, similarly to the case considered in \cite{choi-salem2}. We are in fact allowed slightly more general truncation functions, as will be detailed later on in section \ref{detailed_assumptions}.
\vspace{3mm}

Let  $(\Omega, {\cal F}, ({\cal F}_t)_{t \ge 0} , \Proba)$ be a filtered probability space, and let $\beta, (B^i)_{i \ge 1}$ be independent, respectively  $\mathbb{R}$-valued and $\mathbb{R^d}$-valued $({\cal F}_t)$-brownian motions on $\Omega$, starting from $0$.
Throughout the rest of this paper, we extend our study to a stochastic interacting particle system in $\mathbb{R}^d$ of the general mean-field form
\begin{align}
& dX^{i,N}_t = B[\mu_t^N](X^{i,N}_t) dt + C[\mu_t^N](X^{i,N}_t) \circ d \beta_t + \sigma(X^{i,N}_t) \circ dB^i_t , \label{model}
 \\
 & i \in \{1, \ldots , N \} , \nonumber
\end{align}
where
\begin{align}
\mu_t^N = \frac{1}{N} \sum_{i=1}^N \delta_{X^{i,N}_t}, && B[\mu](x) = \int b(x,y) d\mu(y), && 
C[\mu](x) = \int c(x,y) d\mu(y)
\label{B}
\end{align}
for some coefficients $b, c : \mathbb{R}^d \times \mathbb{R}^d \to \mathbb{R}^d$ and $\sigma : \mathbb{R}^d \to \mathbb{R}^{d \times d}$ . The particles in \myref{model} are subject to two noises of different nature: some individual noise dragged by $B^i_t$ and some common noise dragged by $\beta_t$. The case $c(x,y) = c(x)$ corresponds to a noisy environment (as in \myref{pert2}) whereas the case $c(x,y) = c(x-y)$ corresponds to a noisy interaction (as in \myref{pert3}).

For simplicity purposes, from this point on we choose to only consider a "one-dimensional" common noise $c(x,y) \circ d\beta_t$. It may of course be replaced with a more general $\sum_{k=1}^\infty c_k(x,y) \circ d\beta^k_t$. The results presented in this paper will still hold, provided essentially that the assumptions made here on $c$ are satisfied by all $c_k$, with constants which are square-summable over $k$, as suggested in \myref{ass_phi}.
\vspace{3mm}

In view of usual stochastic mean-field results, it is natural to expect that the limiting equation for the empirical measure $\mu_t^N$ associated to \myref{model} as $N$ goes to infinity is given by
\begin{align}
 d\mu_t + \nabla \cdot ( B[\mu_t] \mu_t ) dt + \nabla \cdot ( C[\mu_t] \mu_t ) \circ d\beta_t + \frac{1}{2} \nabla \cdot ( \text{Tr}(\nabla \sigma \sigma^T ) \mu_t ) dt  
= \frac{1}{2} \nabla \cdot ( \nabla \cdot ( \sigma \sigma^T \mu_t) ) dt
\label{chap2-spde}
\end{align}
where we have used the slight abuse of notation:
\begin{align}
\Big( \text{Tr}(\nabla \sigma \sigma^T) \Big)_i = \text{Tr} ( (\nabla \sigma_i) \sigma^T) 
= \sum_{k, l = 1}^d (\partial_k \sigma_{i,l} ) \sigma_{k,l} .
\label{trace}
\end{align}

Due to the driving noise $\beta_t$ which is common to all particles, \myref{chap2-spde} is an SPDE, so that the limiting measure $(\mu_t)_{t \ge 0}$ is still a stochastic process. The individual noises $\sigma dB^i_t$ are expected to average into the elliptic operator $\frac{1}{2} \nabla \cdot (\nabla \cdot (\sigma \sigma^T . ) )$. 
The first order operator 
$\frac{1}{2} \nabla \cdot ( \text{Tr}(\nabla \sigma \sigma^T ) . )$ only results from the correction from Stratonovich to Itô integration. In the particular case $\sigma(x) \equiv \sigma Id$, we are simply left with
\begin{align*}
& d\mu_t + \nabla \cdot ( B[\mu_t] \mu_t ) dt + \nabla \cdot ( C[\mu_t] \mu_t ) \circ d\beta_t 
 = \frac{\sigma^2}{2} ( \Delta \mu_t ) dt .
\end{align*}

The mean-field limit of the particle system \myref{model} is well known and established when the coefficients $b$, $c$ and $\sigma$ are globally Lipschitz-continuous (see e.g \cite{flandoli}). 
In this article, we want to consider Cucker-Smale perturbations of the form \myref{pert1}, \myref{pert2} and \myref{pert3}. This corresponds to 
$Z_t^{i,N}=(X_t^{i,N},V_t^{i,N})$ satisfying \myref{model} in $\mathbb{R}^{2d}$ with coefficients of the form
\begin{equation}
b((x,v);(y,w)) = \left(
\begin{array}{c}
v \\
\psi(x-y)(w-v)
\end{array}
\right)
\; \; \; \;
c((x,v) ; (y,w) ) =
\left(
\begin{array}{c}
0 \\
\phi(x-y)(w-v)
\end{array}
\right)
\label{chap2-cucker_smale}
\end{equation}
which, when $\psi$ and $\phi$ are globally Lipschitz-continuous, are only locally Lipschitz-continuous.~~\\ 
This leads to additional difficulties compared to the "globally Lipschitz" case. A classical way to deal with such difficulties is to introduce suitable stopping times. In the case considered here, the problem is more difficult since the non-linear terms in equation \myref{chap2-spde} depend on the trajectories of all the particles. This requires to stop every particle at once, leading us to essentially derive estimates on 
\begin{equation}
\sup_{i \in \{1, \ldots N\}} \sup_{t \in [0,T]} |X^{i,N}_t|
\label{the_bound}
\end{equation}
as made clear in Proposition \ref{comparison} and developed in section \ref{chap2-strong}. The bound \myref{the_bound} is the crucial tool in \cite{choi-salem} for instance, where it is derived from a stochastic Gronwall inequality that relies on the simple linear form of the noise. In our case where the noise is more complex, this bound can be obtained through the use of exponential moments for the particles, using a method similar to the one suggested in \cite{bolley}.
This is dealt with in more details in section \ref{section_chara2}.
\vspace{3mm}

Under the assumption that the coefficients $b$, $c$ and $\sigma$ are only locally Lipschitz-continuous and sublinear, we prove the convergence in law (up to a subsequence) of the empirical measure associated to \myref{model} to a weak solution of the limiting SPDE \myref{chap2-spde}. In a more restrictive setting, 
considering only common noise, 
requiring boundedness for $c$ and additional assumptions regarding the growth of local Lipschitz norms of the coefficients, this weak convergence is turned into a strong $L^p(\Omega)$ convergence and the propagation of chaos is established. Precise assumptions and results are stated in section \ref{results} below.

Note that \myref{model} and \myref{chap2-spde} have only been given in the (heuristical) Stratonovich form. In section \ref{ito_form}, we shall determine the corresponding Itô forms and derive a proper definition for solutions of \myref{model} and particularly \myref{chap2-spde} (see Definition \ref{solution}).

\end{subsection}

\begin{subsection}{Main results.} \label{results}

In the rest of this paper, ${\cal P}(E)$ shall denote the set of probability measures on some space $E$. ~~\\
The results presented here along with their proofs involve some considerations regarding ~~\\
Wasserstein spaces 
${\cal P}_p(E)$. 
\begin{define} \label{Wasserstein}
Given $(E,\| . \|)$ a separable Banach space and $p \ge 1$, the $p$th-Wasserstein space
\begin{align*}
{\cal P}_p(E) = \Big\{ \mu \in {\cal P}(E), \; \int_{x \in E} \|x\|^p d\mu(x) < \infty \Big\}
\end{align*}
is equipped with the distance
\begin{align*}
& W_p[ \mu, \nu] = \Big( \inf_{\pi \in \Pi(\mu,\nu)} \int_{x^1, x^2 \in E} \| x^1-x^2 \|^p d\pi(x^1,x^2) \Big)^{1/p}, 
\end{align*}
where
\begin{align*}
\Pi(\mu,\nu) = \Big\{ \pi \in {\cal P}(E^2), \;  \int_{x^2 \in E} \pi(. , dx^2) = \mu \text{ and }
\int_{x^1 \in E} \pi(dx^1 , .)  = \nu
\Big\} .
\end{align*}
\end{define} 
In the rest of this paper, we shall sometimes use the notation $A(z) \lesssim B(z)$ to signify that there exists a constant $C > 0$ independent of the variable $z$ considered such that $A(z) \le C B(z)$ for all $z$.
Defining the Stratonovich corrective terms (see section \ref{ito_form})
\begin{align}
& s_1(x,y,z) = \frac{1}{2} \nabla_x c(x,y) c(x,z) + \nabla_y c(x,y) c(y,z) , \label{s_1}
\\
& S_2(x) = \text{Tr} ( (\nabla \sigma_i) \sigma^T)  \label{S_2}
= \sum_{k, l = 1}^d (\partial_k \sigma_{i,l} ) \sigma_{k,l},
\end{align}
we shall first make the following assumptions on the coefficients of \myref{model}:

\begin{assume}[Sublinearity] \label{hyp1}
\begin{align*}
 &  | b(x,y) | \lesssim  1 + |x| + |y|, \hspace{10mm} |c(x,y)| \lesssim 1+ |x| + |y|,  \hspace{10mm} 
   |\sigma(x)| \lesssim 1+ |x|   \\
 & |s_1(x,y,z) | \lesssim 1 + |x| + |y| + |z|, \hspace{10mm}  |S_2(x)| \lesssim 1+ |x| .
\end{align*}
\end{assume}

\begin{assume}[Locally Lipschitz] \label{hyp2}
\begin{align*}
b,c,\sigma, \nabla c, \nabla \sigma \text{ are locally Lipschitz-continuous}.
\end{align*}
\end{assume}
In this rather general setup, the local Lipschitz-continuity alone is not enough to ensure "standard" estimates of the form
\begin{align*}
\E \Big[ W^2_2 [ \mu_t^N, \mu_t^{M}] \Big] \le C W^2_2[ \mu_0^{N} , \mu_0^{M}]
\end{align*}
and we are not able to establish the "strong" convergence of the particle system. Instead, we rely on compactness arguments to prove the following weak mean-field limit result.

\begin{theorem} \label{chap2-thm1}
Let $T > 0$ and ${\cal C}:=C([0,T] ; \mathbb{R}^d)$ equipped with $\| x \|_\infty = \sup_{t \in [0,T]} |x_t|$ .
\\
Suppose that Assumptions \ref{hyp1} and \ref{hyp2} are satisfied. Let $\mu_0 \in {\cal P}(\mathbb{R}^d)$ such that
$$
\int |x|^{2+\delta} d\mu_0(x) < \infty \text{ for some } \delta > 0.
$$
 Let $(X_t^{i,N})_{t \ge 0}^{i=1, \ldots , N}$ be a solution of \myref{model} and $\mu^N \in {\cal} {\cal P}({\cal C})$ the associated empirical measure. ~~\\
Provided that
$$
  \mu_0^{N} \to \mu_0 \text{ in } {\cal P}_{2}(\mathbb{R}^d) \; \; \text{ and } \; \;
  \sup_N \int |x|^{2 + \delta} d\mu_0^N(x) < \infty ,
$$
there exists a subsequence 
$(\mu^{N'})_{N'}$ such that 
$$
\mu^{N'} \to \mu \text{ in law, in } {\cal P}_2({\cal C})
$$
and $\mu$ is 
a martingale solution of \myref{chap2-spde} (in the sense of \cite{da_prato}, Chapter 8): 
there exists some other probability space $(\widetilde \Omega, \widetilde {\cal F}, \widetilde \Proba)$ equipped with a brownian motion $\widetilde \beta$ such that $\mu$ satisfies the assumptions of Definition \ref{solution} below on $\widetilde \Omega$.
\end{theorem}

This weak convergence can be strengthened into a strong convergence for compactly supported initial measures, under some more restrictive assumptions on the coefficients.

First, we shall only consider the case of common noise (adding individual noises would require additional work, see Remark \ref{work}).
\begin{assume}[Common noise only] \label{hyp_common}
\begin{align*}
\sigma = 0 .
\end{align*}
\end{assume}
In this case, the limiting SPDE \myref{chap2-spde} becomes a stochastic conservation equation:
\begin{align}
d\mu_t + \nabla \cdot ( B[\mu_t] \mu_t ) dt + \nabla \cdot ( C[\mu_t] \mu_t ) \circ d\beta_t 
= 0.
\label{conservation}
\end{align}
Non-linear stochastic conservation equations resembling \myref{conservation} (with local non-linearities) have been studied for instance in \cite{souganidis2}, \cite{souganidis1}.
A solution of \myref{conservation} is naturally expected to be "of the transport form" $\mu = (X^\mu)^* \mu_0$, i.e $\mu$ is given by the push-forward measure of the initial data by the (non-linear) stochastic characteristics
\begin{equation*}
\left\{
\begin{array}{l}
 dX_t^\mu(x) = B[\mu_t](X^\mu_t(x)) dt + C[\mu_t](X^\mu_t(x)) \circ d\beta_t,
\\
 X_0^\mu(x) = x \in \mathbb{R}^d .
 \end{array}
 \right.
\end{equation*}
A precise statement on measures of the transport form is made in Definition \ref{transport_form}. Let us make some additional assumptions on the coefficients:

\begin{assume}[Sublinear drift, bounded diffusion coefficient] \label{hyp1'}
\begin{align*}
  & | b(x,y) | \lesssim  1 + |x| + |y|, \\
  & |s_1(x,y,z) | \lesssim 1 + |x| + |y| + |z|, \\
  & |c(x,y)| \lesssim 1, \nonumber 
\end{align*}
\end{assume}

\begin{assume}[Growth of the local Lipschitz constants] \label{hyp2'} 
\begin{align*}
& |b(x,y) - b(x',y')| \lesssim L_b(x,y,x',y')  \Big( |x-x'| + |y-y|' \Big), \\
& |s_1(x,y,z) - s_1(x',y',z') | \lesssim L_s(x,y,z,x',y',z') \Big( |x-x'| + |y-y'| + |z-z'| \Big),
\\
& |c(x,y) - c(x',y') | \lesssim L_c(x,y,x',y') \Big( |x-x'| + |y-y'| \Big).
\end{align*}
where, for some $\theta \in (0,1)$
\begin{align*}
&  L_b(x,y,x',y') =  1 + |x|^{2 \theta} + |y|^{2 \theta} + |x'|^{2 \theta} + |y'|^{2 \theta},
\\
& L_s(x,y,z,x',y',z') =  1 + |x|^{2 \theta} + |y|^{2 \theta} + |z|^{2 \theta} + |x'|^{2 \theta} + |y'|^{2 \theta} + |z'|^{2 \theta} ,
\\
& L_c(x,y,x',y') = 1 + |x|^\theta + |y|^\theta + |x'|^\theta + |y'|^\theta .
\end{align*}
\end{assume}

\begin{theorem} \label{chap2-thm2} Let $T > 0$, $p \ge 2$ and ${\cal C}:=C([0,T] ; \mathbb{R}^d)$. ~~\\
Suppose that Assumptions \ref{hyp_common}, \ref{hyp1'} and \ref{hyp2'} are satisfied. With the same notations as before, provided that $\mu_0^{N}$  is uniformly supported in some compact set $K \subset \mathbb{R^d}$ and  $\mu_0^{N} \to \mu_0 \text{ in } {\cal P}_p(\mathbb{R}^d)$,
we have the convergence
$$
\mu^N \to \mu \text{ in } L^p(\Omega ; {\cal P}_p({\cal C}) ),
$$
where $\mu = (\mu_t)_{t \in [0,T]}$ is the unique solution of the transport form of \myref{conservation}, in the sense of Definitions \ref{solution} and \ref{transport_form} below.
\end{theorem}
Finally, let us complete this last statement by presenting a result of (conditional) propagation of chaos similar to the one formulated in \cite{flandoli}.

\begin{theorem}  \label{thm3}
Let $T > 0$, $p \ge 2$ and ${\cal C}:=C([0,T] ; \mathbb{R}^d)$.~~\\
Suppose that Assumptions \ref{hyp_common}, \ref{hyp1'} and \ref{hyp2'} are satisfied, and let $({\cal F}^\beta_t)_{t \in [0,T]}$ denote the canonical filtration associated with $\beta$. Given $\mu_0 \in {\cal P}(\mathbb{R}^d)$ supported in some compact set $K \subset \mathbb{R}^d$, let us introduce
\vspace{-2mm}
$$
(\xi_0^i)_{i \ge 1} \text{ i.i.d, ${\cal F}^\beta_0$-measurable, $\mathbb{R}^d$-valued random variables with law $\mu_0$}.
$$
Let $(X_t^{i,N})_{t \ge 0}^{i=1, \ldots , N}$ be the solution of \myref{model} with the initial conditions $X^{i,N}_0 = \xi_0^i$, and let $\mu^N \in {\cal} {\cal P}({\cal C})$ be the associated empirical measure. Then we have the convergence
$$
\mu^N \to \mu \text{ in } L^p(\Omega ; {\cal P}_p({\cal C}) ),
$$
where $\mu = (\mu_t)_{t \in [0,T]}$ is the unique solution of the transport form of \myref{conservation}, in the sense of Definitions \ref{solution} and \ref{transport_form} below. Additionally, for all $r \ge 1$ and $\phi_1, \ldots, \phi_r \in C_b({\cal C})$ we have
$$
\E \left[ \phi_1(X^{1,N}) \ldots \phi_r(X^{r,N}) | {\cal F}_T^\beta \right] \to \prod_{i=1}^r \langle \phi_i, \mu \rangle \text{ in } L^1(\Omega) . \label{chaos}
$$
Finally, for all $i \ge 1$, let $X^i$ be the solution of
$$
\left\{
\begin{array}{l}
 dX_t^i = B[\mu_t](X^i_t) dt + C[\mu_t](X^i_t) \circ d\beta_t,
\\
 X_0^i = \xi_0^i.
 \end{array}
 \right.
$$
Then the limiting measure $\mu \in {\cal P}({\cal C})$ is a version of the conditional law 
${\cal L}(X^i | {\cal F}_T^\beta )$ and we have the convergence
$$
X^{i,N} \to X^i \text{ in } L^p(\Omega ; {\cal C}).
$$

\end{theorem}

\end{subsection}

\begin{subsection}{Itô form.} \label{ito_form}

Let us now determine the proper Itô form expressions of \myref{model} and \myref{chap2-spde}. Itô's formula gives
\begin{align*}
d \Big[ C[\mu_t^N] (X^{i,N}_t) \Big] & = \frac{1}{N} \sum_j d \Big[ c(X^{i,N}_t,X^{j,N}_t) \Big]
\\
& = \Big( \frac{1}{N} \sum_j \nabla_x c(X^{i,N}_t,X^{j,N}_t) C[\mu^N_t](X^{i,N}_t)  + \nabla_y c(X^{i,N}_t,X^{j,N}_t) C[\mu^N_t](X^{j,N}_t)  \Big) d\beta_t  \\
& \hspace{15mm} + dV^{i,j}_t + dM^{i,j}_t
\end{align*}
where $V^{i,j}$ is a process with bounded variation and
\begin{align*}
dM^{i,j}_t = \frac{1}{N} \sum_j \left( \nabla_x c(X^{i,N}_t, X^{j,N}_t) \sigma(X^{i,N}_t) dB^i_t +
\nabla_y c(X^{i,N}_t, X^{j,N}_t) \sigma(X^{j,N}_t) dB^j_t 
   \right) .
\end{align*}
It follows that the correction from Stratonovich to Itô is given by
\begin{align*}
& C[\mu_t^N](X^{i,N}_t) \circ d\beta_t = C[\mu_t^N](X^{i,N}_t) d\beta_t + S_1[\mu_t^N](X^{i,N}_t) dt \;
\end{align*}
with
\begin{align}
& S_1[\mu](x) = \int \int s_1(x,y,z) d\mu(y) d\mu(z) \label{S_1}
\end{align}
where
\begin{align*}
s_1(x,y,z) = \frac{1}{2} \nabla_x c(x,y) c(x,z) + \nabla_y c(x,y) c(y,z)
\end{align*}
as defined in \myref{s_1}. Similarly, the correction for the individual noise is given by
\begin{align*}
& \sigma(X^{i,N}_t) \circ dB^i_t = \sigma(X^{i,N}_t) dB^i_t + S_2(X^{i,N}_t) dt
\end{align*}
with
\begin{align*}
& S_2(x) = \frac{1}{2} \text{Tr}(\nabla \sigma(x) \sigma^T(x))
\end{align*}
as defined in \myref{S_2}. We may now rewrite the particle system \myref{model} as
\begin{align}
& dX^{i,N}_t = \Big( B[\mu_t^N](X^{i,N}_t) + S[\mu^N_t](X^{i,N}_t) \Big) dt + C[\mu_t^N](X^{i,N}_t) d\beta_t + \sigma(X^{i,N}_t) dB^i_t , \label{ito}
\end{align}
where
\begin{align*}
S[\mu](x) = S_1[\mu](x) + S_2(x) \text{ is defined in \myref{S_1} and \myref {S_2} } . 
\end{align*}
~~\\
As for the SPDE \myref{chap2-spde}, it is to be understood in the following weak sense: for any $\psi \in C^2_c(\mathbb{R}^d)$,
\begin{align*}
d \langle \psi , \mu_t \rangle = \langle ( B[\mu_t] + \frac{1}{2} Tr(\nabla \sigma \sigma^T) ) \cdot \nabla \psi , \mu_t \rangle dt + 
\langle C[\mu_t] \cdot \nabla \psi , \mu_t \rangle \circ d\beta_t + \frac{1}{2} \langle \text{Tr}(\sigma (\nabla^2 \psi) \sigma^T), \mu_t \rangle dt .
\end{align*}
Let us determine the correction corresponding to the Stratonovich term. We have
\begin{align*}
d \Big[ \langle C[\mu_t] \cdot \nabla \psi , \mu_t \rangle \Big]
 =  \langle C[\mu_t] \cdot \nabla \psi , d\mu_t \rangle + \langle d \Big[ C[\mu_t] \cdot \nabla \psi \Big] , \mu_t \rangle .
\end{align*}
On one hand,
\begin{align*}
\langle C[\mu_t] \cdot \nabla \psi , d\mu_t \rangle & = \langle C[\mu_t] \cdot \nabla( C[\mu_t] \cdot \nabla \psi) , \mu_t \rangle d\beta_t + dV^{(1)}_t 
\\
& = \Big( \langle \left( \nabla C[\mu_t] C[\mu_t] \right) \cdot \nabla \psi, \mu_t \rangle +
\langle \nabla^2 \psi \cdot C[\mu_t]^{\otimes 2} , \mu_t \rangle
 \Big) d\beta_t + dV^{(1)}_t 
\end{align*} 
where $V^{(1)}$ is a process with bounded variation.
On the other hand,
\begin{align*}
C[\mu_t] (x) \cdot \nabla \psi(x) = \int \phi(y) d\mu_t(y) = \langle \phi , \mu_t \rangle
\text{ with } \phi(y) = c(x,y) \cdot \nabla \psi(x)
\end{align*}
so that
\begin{align*}
d \Big[ C[\mu_t] \cdot \nabla \psi \Big](x)  & = d \langle \phi, \mu_t \rangle = \langle C[\mu_t] \cdot \nabla \phi , \mu_t \rangle d\beta_t + dV^{(2)}_t(x) \\
 & = \left( \left( \int \nabla_y c(x,y) C[\mu_t](y) d\mu_t(y) \right) \cdot \nabla \psi(x) \right) d\beta_t 
+ dV^{(2)}_t(x) 
\end{align*}
where $V^{(2)}(x)$ is a process with bounded variation.
Combining both expressions, we are led to
\begin{align*}
d \Big[ \langle C[\mu_t] \cdot \nabla \psi , \mu_t \rangle \Big] = \langle 2 S_1[\mu_t] \cdot \nabla \psi + \nabla^2 \psi \cdot (C[\mu_t]^{\otimes 2}) , \mu_t \rangle d\beta_t + dU_t
\end{align*}
where $U_t = \int_0^t \left( dV^{(1)}_s + \langle dV^{(2)}_s, \mu_s \rangle \right)$ is a process with bounded variation.
The correction is therefore given by
\begin{align*}
\langle C[\mu_t] \cdot \nabla \psi , \mu_t \rangle \circ d\beta_t =
\langle C[\mu_t] \cdot \nabla \psi , \mu_t \rangle d\beta_t
+ \Big( \langle S_1[\mu_t] \cdot \nabla \psi, \mu_t \rangle +  \frac{1}{2} \langle \nabla^2 \psi \cdot (C[\mu_t]^{\otimes 2}) , \mu_t \rangle \Big) dt .
\end{align*}
Consequently, the Itô form corresponding to the SPDE \myref{chap2-spde} is exactly
\begin{align}
& d\mu_t + \nabla \cdot \Big( ( B[\mu_t] + S[\mu_t] ) \mu_t \Big) dt + \nabla \cdot (C[\mu_t] \mu_t) d\beta_t
=
\frac{1}{2} \nabla \cdot \nabla \cdot \Big( (\sigma \sigma^T + C[\mu_t] C[\mu_t]^T) \mu_t \Big) dt
\label{spde_ito}
\end{align}
with $S[\mu_t]$ as in \myref{ito}. This allows us to precisely define the notion of solution for \myref{chap2-spde}.

\begin{define} 
\label{solution}
Let $(\Omega, {\cal F}, ({\cal F}_t) , \Proba)$ be a filtered probability space equipped with an $({\cal F}_t)$-brownian motion $\beta$. Let $\mu_0 \in {\cal P}(\mathbb{R}^d)$. 

A measure-valued process $\mu = (\mu_t)_{t \in [0,T]} : \Omega \to {\cal P}(\mathbb{R}^d)^{[0,T]}$ is said to be a solution of the SPDE \myref{chap2-spde} (or equivalently \myref{spde_ito}) with initial value $\mu_0$  when for all $\psi \in C^2_c(\mathbb{R}^d)$, the process $(\langle \psi,\mu_t \rangle)_{t \in [0,T]}$ is adapted with a continuous version and satisfies
\begin{align}
\langle \psi, \mu_t \rangle = \langle \psi, \mu_0 \rangle + \int_0^t \langle \big( B[\mu_s] + S[\mu_s]\big) \cdot \nabla \psi + {\cal A}[\mu_s] \psi , \mu_s \rangle ds + \int_0^t \langle C[\mu_s] \cdot \nabla \psi , \mu_s \rangle d\beta_s,
\label{eq_psi}
\end{align}
where $S[\mu]$ is defined in \myref{ito} and the second order operator ${\cal A}[\mu]$ is given by
\begin{align}
{\cal A}[\mu] \psi = \frac{1}{2} \sum_{i,j} \Big( \sum_k \sigma_{i,k} \sigma_{j,k} + C_i[\mu] C_j[\mu]  \Big)  \partial^2_{i,j} \psi .
\label{calA}
\end{align}

\end{define}

\begin{remark}
Comparing the particle system \myref{ito} and the SPDE \myref{spde_ito} expressed in Itô form, we see that the correction from Stratonovich to Itô integration adds some "virtual" interaction kernel $S[\mu]$ to the system. On the SPDE \myref{spde_ito}, it additionally results in the  operator ${\cal A}[\mu]$ which is of order $2$ and consequently is not "visible" on the particle system \myref{ito}.
\end{remark}

\end{subsection}

\end{section}


\begin{section}{Weak mean-field convergence} \label{chap2-2}

\begin{subsection}{Properties of the coefficients.}
In the entirety of section \ref{chap2-2}, we shall assume that Assumptions \ref{hyp1} and \ref{hyp2} are satisfied.
\vspace{3mm}

Assumption \ref{hyp2} guarantees that the coefficients of the SDE system expressed in Itô form \myref{ito} are locally Lipschitz-continuous, which classically provides the local existence and uniqueness of solutions. The sublinearity Assumption \ref{hyp1} immediately results in
\begin{align}
|B[\mu](x)|, \; \; |C[\mu](x)|,  \; \; | S[\mu](x) |  \lesssim  \Big(1+  |x| + \int |y| d\mu(y) \Big) \label{sublinear_B} .
\end{align}
Of course, Assumptions \ref{hyp1} and \ref{hyp2} are satisfied in the classical "globally-Lipschitz" setup when
\begin{align*}
| \nabla b |, |\nabla c|, | \nabla \sigma | \lesssim 1, \; \; \; 
\text{$\nabla c$ and $\nabla \sigma$ locally Lipschitz-continuous}.
\end{align*}
 Most importantly, we are indeed allowed to consider coefficients with the Cucker-Smale form: let $b$ and $c$ be given by \myref{chap2-cucker_smale}. Assuming that $\psi$, $\phi$ are bounded and locally Lipschitz-continuous, $b$ and $c$ are clearly sublinear. Moreover, a simple calculation gives, with $z_i = (x_i,v_i)$,
\begin{equation*}
s_1(z_1,z_2,z_2) = \left(
\begin{array}{c}
0
\\
- \phi(x_1-x_2) \phi(x_1-x_3) (v_3-v_1) + \phi(x_1-x_2) \phi(x_2-x_3) (v_3-v_2)
\end{array}
\right)
\end{equation*}
which is sublinear as well. 
\end{subsection}

\begin{subsection}{Estimates for the particle system.}
Firstly, Assumption \ref{hyp1} naturally guarantees some moment estimates for the solutions of \myref{model}.

\begin{proposition}[Moment estimates, global existence] \label{chap2-moments}
~~\\
Let $T > 0$, $q \ge 2$ and $\mu_0^N = \frac{1}{N} \sum_i \delta_{X_0^{i,N}}$ be such that $\int |x|^q d\mu_0^N(x) < \infty$. Then the SDE system \myref{model} (or equivalently \myref{ito}) has a unique solution defined on $[0,T]$, which satisfies, 
\begin{align}
& \E \Big[ \sup_{t \in [0,T]} \int |x|^q d\mu^N_t(x) \Big] \lesssim 1 + \int |x|^q d\mu_0^N(x)
\end{align}
and for all $i \in \{1, ... , N \}$,
\begin{align}
& \E \Big[ \sup_{t \in [0,T]} |X^{i,N}_t|^q \Big] \lesssim  1 + \Big( |X^{i,N}_0|^q + \int |x|^q d\mu_0^N(x)  \Big).
\label{big_bound}
\end{align}
The constants involved in $\lesssim$ depend on $T$ and $q$ only.
\end{proposition} 

\begin{proof}[Proof]
The assumptions guarantee that the coefficients of the SDE \myref{ito} are locally Lipschitz-continuous, which provides the local existence and uniqueness of the solution. To simplify the notation, we shall consider that all stochastic integrals are well defined: for a more rigorous framework, one should consider the solution of the truncated equations with a suitable stopping time ; classically, estimate \myref{big_bound} (uniform on the truncation) then ensures that the solution is globally defined. Using \myref{sublinear_B}, one can write
\begin{align}
|X_t^{i,N}|^q & \lesssim  |X_0^{i,N}|^q + \int_0^t ( |B[\mu_s^N](X_s^{i,N})|^q + |S[\mu_s^N](X_s^{i,N})|^q  )ds + |M^i_t|^q \nonumber \\
& \lesssim  1 + |X_0^{i,N}|^q + \int_0^t ( |X_s^{i,N}|^q + \int |x|^q d\mu_s^N ) ds + |M_t^i|^q 
\label{nobar}
\end{align}
where $M^i_t = \int_0^t C[\mu_s^N](X^{i,N}_s) d \beta_s + \int_0^t \sigma(X^{i,N}_s) dB^i_s$. Taking the mean over $i$, and letting 
$$\overline{|X_t|^q}~=~\int |x|^q d\mu^N_t$$
we are led to
\begin{align*}
\overline{|X_t|^q}  \lesssim 1 + \overline{|X_0|^q} + \int_0^t \overline{|X_s|^q} ds + \frac{1}{N} \sum_i |M^i_t|^q 
\end{align*}
and therefore,
\begin{align}
\sup_{\sigma \in [0,t]} \overline{|X_\sigma|^q} \lesssim 1 +
\overline{|X_0|^q} + \int_0^t \sup_{\sigma \in [0,s]} \overline{|X_\sigma|^q} ds 
+ \frac{1}{N} \sum_i \sup_{\sigma \in [0,t]} |M^i_\sigma|^q .
\label{bar}
\end{align}
Burkholder-Davis-Gundy's inequality from \cite{bdg} states that $\E \Big[ \sup_{\sigma \in [0,t]} |M^i_\sigma|^q \Big]_t \lesssim \E\left( \Big[ M^i \Big]_t^{q/2} \right)$.~~\\
Using \myref{sublinear_B},
\begin{align*}
\Big[ M^i \Big]_t =  \int_0^t |C[\mu_s^N](X_s^{i,N})|^2 + |\sigma(X_s^{i,N})|^2 ds \lesssim 1 + \int_0^t |X^{i,N}_s|^2 + \overline{|X_s|^2} ds
\end{align*}
hence $ \Big[ M^i \Big]_t^{q/2} \lesssim 1 + \int_0^t|X^{i,N}_s|^q + \overline{|X_s|^q} ds$.  Coming back to \myref{bar},
\begin{align}
\E \Big[ \sup_{\sigma \in [0,t]} \overline{|X_\sigma|^q} \Big] \lesssim 1 + \overline{|X_0|^q} + \int_0^t \E \Big[ \sup_{\sigma \in [0,s]} \overline{|X_\sigma|^q} \Big] ds 
\end{align}
and we use Grönwall's Lemma to get the first estimate of Proposition \ref{chap2-moments}. We can now get back to \myref{nobar} to get
\begin{align*}
\sup_{\sigma \in [0,t]} | X_t^{i,N}|^q \lesssim 1 + |X_0^{i,N}|^q + \int_0^t \sup_{\sigma \in [0,s]} |X_s^{i,N}|^q ds +  \sup_{\sigma \in [0,T]} \overline{|X_\sigma|^q} + \sup_{\sigma \in [0,t]} |M^i_t|^q.
\end{align*}
Using the previously established estimate and Burkholder-Davis-Gundy's inequality once again,
\begin{align*}
\E \left[ \sup_{\sigma \in [0,t]} | X_t^{i,N}|^q \right] \lesssim  \left( 1+  |X_0^{i,N}|^q + \overline{|X_0|^q} \right) + 
\int_0^t \E \left[ \sup_{\sigma \in [0,s]} |X_s^{i,N}|^q \right] ds
\end{align*}
and we may apply Grönwall's Lemma to obtain the second estimate of Proposition \ref{chap2-moments}.
\end{proof}

\begin{remark}
Given some $\psi \in C^2(\mathbb{R}^d)$ with $| \nabla \psi|, |\nabla^2 \psi| \lesssim1$, Itô's formula gives
\begin{align*}
d\psi(X^{i,N}_t) & = \nabla \psi(X^{i,N}_t) \cdot (B[\mu_t^N] + S[\mu_t^N])(X^{i,N}_t) dt + \nabla \psi(X^{i,N}_t) \cdot C[\mu_t^N](X^{i,N}_t) d\beta_t 
\\ & + \nabla \psi(X^{i,N}_t) \cdot \sigma(X^{i,N}_t) dB^i_t + {\cal A}[\mu_t^N](X^{i,N}_t) dt,
\end{align*}
hence taking the mean in $i \in \{1, \ldots, N\}$ we are led to
\begin{align}
\langle \psi, \mu^N_t \rangle & = \langle \psi, \mu^N_0 \rangle + \int_0^t \langle \big( B[\mu^N_s] + S[\mu^N_s]\big) \cdot \nabla \psi + {\cal A}[\mu^N_s] \psi , \mu^N_s \rangle ds + \int_0^t \langle C[\mu^N_s] \cdot \nabla \psi , \mu^N_s \rangle d\beta_s
\nonumber \\
& + \frac{1}{N} \sum_i \int_0^t \nabla \psi(X^{i,N}_s) \cdot \Big( \sigma(X^{i,N}_s) dB^i_s \Big) . \label{this}
\end{align}
Given the bounds on $\E[ \int |x|^2 d\mu_t^N(x) ]$, it is easy to see that the stochastic integrals involved are continuous martingales. Aside from the last term, which is expected to vanish as $N$ goes to infinity, this is exactly the SPDE \myref{spde_ito}.
\end{remark}

Let us now establish some estimates regarding the regularity of solutions of \myref{model}.

\begin{proposition}[Kolmogorov continuity for the particle system] \label{kolmogorov} ~~\\
Let $T > 0$, $q \ge 2$ and $\mu_0^N = \frac{1}{N} \sum_i \delta_{X_0^{i,N}}$ be such that $\int |x|^q d\mu_0^N(x) < \infty$. ~~\\ The following estimate holds uniformly for $t,s \in [0,T]$
\begin{align*}
& \frac{1}{N} \sum_i \E |X_t^{i,N} - X_s^{i,N} |^q \lesssim \Big( 1 + \int |x|^q d\mu_0^N \Big)  |t-s|^{q/2} .
\end{align*}
The constant involved in $\lesssim$ depends on $T$ and $q$ only.
\end{proposition}

\begin{proof}[Proof]
Again, one can write
\begin{align*}
|X_t^{i,N} - X_s^{i,N}|^q & \lesssim  \Big| \int_s^t (B[\mu_\sigma^N] + S[\mu_\sigma^N])(X_\sigma^{i,N}) d\sigma \Big|^q 
+ |M_t^i - M_s^i |^q 
\\
& \lesssim  |t-s|^{q-1} \int_s^t ( |B[\mu_\sigma^N](X^{i,N}_\sigma)|^q + |S[\mu_\sigma^N](X^{i,N}_\sigma)|^q ) d\sigma 
+ |M_t^i - M_s^i |^q ,
\end{align*}
hence using the estimates from Proposition \ref{chap2-moments},
\begin{align*}
\frac{1}{N} \sum_i \E |X_t^{i,N} - X_s^{i,N}|^q & \lesssim  |t-s|^q  \left( 1 + \int |x|^q d\mu_0^N \right) + 
\frac{1}{N} \sum_i \E |M_t^i - M_s^i |^q  .
\end{align*}
Burkholder-Davis-Gundy's inequality gives
\begin{align*}
\E |M_t^i - M_s^i |^q & \lesssim \E \Big| \Big[ M^i \Big]_t - \Big[ M^i \Big]_s \Big|^{q/2} \\
& \lesssim \E \Big| \int_s^t 1 + \overline{| X^{N}_\sigma|^2} + |X^{i,N}_\sigma|^2 d\sigma \Big|^{q/2}
\lesssim  |t-s|^{q/2-1} \int_s^t 1 + \E [\overline{| X^{N}_\sigma|^q} + |X^{i,N}_\sigma|^q ] d\sigma 
\end{align*}
since $q/2 \ge 1$, hence
\begin{align*}
\frac{1}{N} \sum_i \E |M_t^i - M_s^i |^q  \lesssim |t-s|^{q/2} \left( 1 + \int |x|^q d\mu_0^N \right)
\end{align*}
which concludes the proof
\end{proof}

\end{subsection}

\begin{subsection}{Tightness of measure-valued random variables.} \label{tightness}

In this subsection, we state general results regarding the tightness of random measures, which we shall later apply in our special case.
\vspace{3mm}

Let $(E, \| . \|)$ be a separable Banach space. The space ${\cal P}(E)$ of probability measures on $E$ is equipped with the topology of the weak convergence. More precisely, we shall consider that ${\cal P}(E)$ is equipped with the Lévy-Prokhorov metric, which also makes it a polish space (one may refer to \cite{billing} for details regarding this topology and tightness in general).

\begin{define}
For a random measure $\mu : \Omega \to {\cal P}(E)$, we define the intensity $I(\mu)$ of $\mu$ by
\begin{align*}
\forall f : E \to \mathbb{R} \text{ measurable, bounded },  \; \; \langle f, I(\mu) \rangle = \E[ \langle f , \mu \rangle],
\end{align*}
that is $I(\mu) \in {\cal P}(E)$ is a deterministic probability measure on $E$. 
\end{define}

The following result (mentioned e.g in \cite{chaos} p178) establishes a link between the relative compactness in law of $\mu$ and the tightness of its intensity measure $I(\mu)$.

\begin{proposition} \label{intensity} ~~\\
For a sequence $(\mu^N)_{N \ge 1}$  of random measures on $E$, the two following statements are equivalent.
\begin{enumerate}[i)]
\item The sequence of ${\cal P}(E)$-valued random variables $(\mu^N)_N$ is tight.
\item The sequence $(I(\mu^N))_N$ of measures on $E$ is tight.
\end{enumerate}
\end{proposition}

\begin{proof}[Proof]
Firstly $i)$ clearly implies $ii)$ since
$$\left[ \mu^N \to \mu \text{ in } {\cal P}(E) \text{ weakly on } \Omega \right] \text{ implies } 
\left[ I(\mu^N) \to I(\mu) \text{ weakly on E} \right] .
$$
Let us assume $ii)$: we introduce a sequence $(C_m)_{m \ge 1}$ of compacts of $E$ such that
\begin{align*}
\forall m \ge 1, \; \; \sup_N I(\mu^N)(C_m^c) \le 4^{-m} .
\end{align*}
For a given $\varepsilon > 0$, let us define 
\begin{align}
K_\varepsilon = \left\{ \mu \in {\cal P}(E), \; \forall m \ge 1, \; \mu(C_m^c) \le \varepsilon^{-1} 2^{-m} \right\} .
\end{align}
Prokhorov's theorem on tightness states that $K_\varepsilon$ is a compact of ${\cal P}(E)$ equipped with the Lévy-Prokhorov metric. Now, for all $N \ge 1$, using simply Markov's inequality,
\begin{align*}
\Proba[ \mu^N \notin K_\varepsilon ] \le \sum_m \Proba[ \mu^N(C_m^c) > \varepsilon^{-1} 2^{-m} ]
\le \varepsilon \sum_m I(\mu^N)(C_m^c) 2^m \le \varepsilon .
\end{align*}
\end{proof}
Let us extend this reasoning to the Wasserstein space ${\cal P}_2(E)$ recalled in Definition \ref{Wasserstein}. ~~\\
Firstly, the following convergence criteria is well known.

\begin{proposition} \label{criteria} ~~\\
For a sequence of measure $(\mu^N)_N \in {\cal P}_2(E)^\mathbb{N}$, the following statements are equivalent:
\begin{enumerate}[i)]
\item 
$\mu^N \to \mu \text{ in } {\cal P}_2(E)$
\item 
$\mu^N \to \mu \text{ in } {\cal P}(E) \text{ and } \; \displaystyle{ \limsup_N \int_{\|x\| > R} \| x \|^2 d\mu^N(x) \xrightarrow[R \to \infty]{} 0  }
$
\item
$\langle \phi, \mu^N \rangle \to \langle \phi, \mu \rangle \text{ for all  $\phi$ continuous such that } \phi(x) \le C(1 + |x|^2).
$
\end{enumerate}
\end{proposition}
One could refer to \cite{villani}, section 6 (Theorem 6.9) for a proof.
This immediately results in
\begin{corollary}[Compact subsets of ${\cal P}_2(E)$] ~~\\ \label{compact}
A subset $A \subset {\cal P}_2(E)$ is relatively compact if and only if
\begin{itemize}
\item The family of measures $(\mu)_{\mu \in A}$ is tight
\item $\displaystyle{ \sup_{\mu \in A} \Big[ \int_{\|x\| > R} \| x \|^2 d\mu(x) \Big] \to 0 }$ as $R \to \infty$.
\end{itemize}
\end{corollary}

We can now state the following criteria for the relative compactness in law in ${\cal P}_2(E)$, which we conveniently express through the means of a Skorokhod representation theorem.
\begin{proposition} \label{bridge}
Let $\mu^N : \Omega \to {\cal P}_2(E)$, $N \ge 1$, be a sequence of random measures on $E$. ~~\\
The following statements are equivalent:
\begin{enumerate}[i)]
\item $\{ I(\mu^N), N \ge 1 \}$ is relatively compact in ${\cal P}_2(E)$
\item The family of measures $(I(\mu^N))_N$ is tight and 
$\displaystyle{ \sup_N \int_{\|x\| > R} \| x \|^2 dI(\mu^N)(x) \xrightarrow[R \to \infty]{} 0 }$
\item Out of any subsequence of $(\mu^N)_N$, one can extract a subsequence $(\mu^{N'})_{N'}$ satisfying the following : there exists some probability space $(\widetilde \Omega, \widetilde {\cal F}, \widetilde{\Proba})$ and random variables ~~\\
$\widetilde \mu^{N'}, \widetilde \mu : \widetilde \Omega \to {\cal P}_2(E)$ such that
\begin{align*}
& \forall N', \; \; \widetilde \mu^{N'} \sim \mu^{N'} \text{ in law } \\
& W_2[\widetilde \mu^{N'} , \widetilde \mu] \xrightarrow[N' \to \infty]{} 0 \text{ $\widetilde \Proba$-a.s and in } L^2(\widetilde \Omega).
\end{align*}
\end{enumerate}
\end{proposition}

\begin{remark}
Note that the statement expressed in point \textit{iii)} is in fact slightly stronger than the tightness of the ${\cal P}_2(E)$-valued random variables $(\mu^N)_N$ since, up to a change of probability space, we are able to obtain a convergence in $L^2(\widetilde \Omega ; {\cal P}_2(E))$.
\end{remark}

\begin{proof}[Proof of Proposition \ref{bridge}]

The equivalence between $i)$ and $ii)$ is exactly stated in Corollary \ref{compact}. ~~\\ 
Assuming  $ii)$, we can introduce $(C_m)_{m \ge 1}$ compacts of $E$ and $(R_m)_m$ with $R_m \to \infty$ such that for all $m \ge 1$,
\begin{align*}
& \sup_N I(\mu^N)(C_m^c) \le 4^{-m} ,
& \sup_N  \int_{\| x \| > R_m} \| x \|^2 dI(\mu^N)(x) \le 4^{-m} .
\end{align*}
For a given $\varepsilon > 0$, let us define
\begin{align*}
K_{\varepsilon} = \left\{ \mu \in {\cal P}_2(E), \; \; \forall m \ge 1, \; \mu(C_m^c) \le \varepsilon^{-1} 2^{-m} , \; \;
\int_{ \| x \| > R_m } \| x \|^2 d\mu(x) \le \varepsilon^{-1} 2^{-m} \right\}
\end{align*}
which is a relatively compact subset of ${\cal P}_2(E)$ by Corollary \ref{compact}. Then, using Markov's inequality
\begin{align*}
\Proba[\mu^N \notin K_\varepsilon] & \le \sum_{m \ge 1} \Proba \Big[ \mu^N(C^c_m) > \varepsilon^{-1} 2^{-m} \Big] + 
\Proba\Big[ \int_{\| x \| > R_m} \| x \|^2 d\mu^N(x) > \varepsilon^{-1} 2^{-m} \Big]
\\
& \le  \varepsilon \sum_{m \ge 1} 2^m I(\mu^N)(C^c_m) + 2^m \int_{\| x \| > R_m} \| x \|^2 dI(\mu^N)(x)  \le 2 \varepsilon
\end{align*}
which proves the tightness of the ${\cal P}_2(E)$-valued random variables $(\mu^N)_N$. Let us introduce a subsequence $(\mu^{N'})_{N'}$ which converges in law to some $\mu \in {\cal P}_2(E)$. Applying Skorkhod's representation theorem on the polish space ${\cal P}_2(E)$ we get, on some probability space $(\widetilde \Omega, \widetilde {\cal F})$, $\widetilde \mu^{N'} \to \widetilde \mu$ a.s in ${\cal P}_2(E)$.~~\\
To conclude regarding the convergence in $L^2(\widetilde \Omega ; {\cal P}_2(E))$, it suffices to show that $W_2^2[\widetilde \mu^{N'} , \widetilde \mu]$ is uniformly integrable in $N'$. 
To this purpose, one can simply write 
\begin{align}
W_2^2[\widetilde \mu^{N'}, \widetilde \mu] \lesssim \int \| x \|^2 d \widetilde \mu^{N'} + \int \| x \|^2 d \widetilde \mu
\label{separate}
\end{align}
and note that for all $R, M > 0$,
\begin{align*}
\widetilde{\E} \Big[ \Big(\int \| x \|^2 d \widetilde \mu^{N'} \Big) \mathbf{1}_{ \{ \int \| x \|^2 d \widetilde \mu^{N'}  > R \} } \Big]
& \le \widetilde{\E} \Big[ \int_{ \| x \|^2 > M} \| x \|^2 d \widetilde \mu^{N'} \Big] + 
M \widetilde{\Proba} \Big[ \int \| x \|^2 d \widetilde \mu^{N'}  > R \Big] 
\\
& \le \sup_N \int_{\| x \|^2 > M} \| x \|^2 dI(\mu^N) + \frac{M}{R} \sup_N \int \| x \|^2 dI(\mu^N).
\end{align*}
This shows that for all $M > 0$,
\begin{align*}
\limsup_{R \to \infty} \sup_N 
\widetilde{\E} \Big[ \int \| x \|^2 d \widetilde \mu^{N'} \mathbf{1}_{ \{ \int \| x \|^2 d \widetilde \mu^{N'}  > R \} } \Big]
\le  \sup_N \int_{\| x \|^2 > M} \| x \|^2 dI(\mu^N)  \to 0 \text{ as } M \to \infty,
\end{align*}
hence the first term in \myref{separate} is uniformly integrable. As for the second one, a use of Fatou's lemma gives $\widetilde{\E} \Big[ \int \| x \|^2 d\widetilde \mu \Big] \le \sup_N \int \|x\|^2 dI(\mu^N) < \infty$. We have proved $iii)$.

Finally, let us show that $iii)$ implies $i)$. Let us introduce $\widetilde \pi(\omega)$ an optimal plan between $\widetilde \mu^{N'} (\omega)$ and $\widetilde \mu(\omega)$, that is
\begin{align*}
W_2^2[\widetilde \mu^{N'} , \widetilde \mu ](\omega) = \int_{x^1,x^2 \in E} \|x^1 - x^2 \|^2 d\widetilde \pi(\omega)(x^1,x^2) .
\end{align*}
 Note that such a coupling exists, and can indeed be selected to be measurable, see \cite{villani}, Theorem 4.1 and Corollary 5.22. 
Then $\widetilde \pi \in \Pi[ \widetilde \mu^{N'} , \widetilde \mu ]$ (for every $\omega \in \widetilde \Omega$) and it is clear from the definition that
\begin{align*}
I(\widetilde \pi) \in \Pi[ I(\widetilde \mu^{N'}), I(\widetilde \mu)] = \Pi[ I( \mu^{N'}) , I( \mu) ] .
\end{align*} 
Therefore it follows that
\begin{align*}
W^2_2[ I(\mu^{N'}) , I(\mu) ] & \le \int_{x^1,x^2 \in E} \| x^1 - x^2 \|^2 dI(\widetilde \pi)(x^1,x^2)  \\
& = \widetilde{ \E } \Big[ \int_{x^1, x^2 \in E} \| x^1 - x^2 \|^2 d \widetilde \pi(x^1,x^2) \Big]
= \E \Big[ W^2_2 [ \widetilde \mu^{N'} , \widetilde \mu ] \Big] \to 0 \text{ as } N' \to \infty.
\end{align*}
\end{proof}

\end{subsection}

\begin{subsection}{Proof of the weak convergence.}

We will now prove the result stated in Theorem \ref{chap2-thm1}. Consider $\mu_0 \in {\cal P}(\mathbb{R}^d)$ satisfying
\begin{align*}
\int |x|^{2+\delta} d\mu_0(x) < \infty \text{ for some } \delta > 0
\end{align*}
and a sequence of empirical measures $(\mu_0^N = \frac{1}{N} \sum_{i=1}^N \delta_{x_0^{i,N}})_N$ such that
\begin{align*}
\mu_0^N \to \mu_0 \text{ in } {\cal P}_{2}(\mathbb{R}^d), \; \; \; \; \sup_N \int |x|^{2+\delta} d\mu_0^N(x) < \infty .
\end{align*}
Let $(X_t^{i,N})_{t \in [0,T]}$, $i \in \{1, ... , N\}$ be the solution of \myref{ito} with intial data $X_0^{i,N} = x_0^{i,N}$. We shall look at these processes as random variables taking values in the (separable Banach) space of continuous functions:
\begin{equation}
\begin{array}{r l l}
X^{i,N} : \Omega & \longrightarrow & {\cal C}= C([0,T] ; \mathbb{R}^{d}) \\
\omega & \longmapsto & (X^{i,N}(\omega))_{t \in [0,T]}
\end{array} 
\end{equation}
where ${\cal C}$ is naturally equipped with the norm $\| x \|_\infty = \sup_{t \in [0,T]} |x_t|$. ~~\\ The associated empirical measure $\mu^N = \frac{1}{N} \sum_{i = 1}^N \delta_{X^{i,N}}$  is hence seen as a random element of ${\cal P}({\cal C})$. Its intensity measure is given by
\begin{align}
I(\mu^N)(A) = \frac{1}{N} \sum_{i=1}^N \Proba(X^{i,N} \in A), \; \; \; A \in {\cal B}({\cal C}).
\end{align}
Let us verify the assumptions of Proposition \ref{bridge} to establish the compactness in law in ${\cal P}_2({\cal C})$:
\begin{enumerate}
\item Firstly,
$
 \sup_N \int_{x \in {\cal C}} |x_0|^{2} dI(\mu^N) = \sup_N \int_{x \in \mathbb{R}^d} |x|^{2} d\mu_0^N < \infty
$
and Proposition \ref{kolmogorov} gives 
\begin{align*}
\sup_N \int_{x \in {\cal C}} |x_t - x_s|^{2+\delta} dI(\mu^N) & = \sup_N \frac{1}{N} \sum_i \E |X^{i,N}_t - X^{i,N}_s |^{2+\delta}
\\ &
\lesssim   \Big( 1+ \sup_N \int |x|^{2+\delta} d\mu_0^N \Big) |t-s|^{1 + \delta/2} .
\end{align*}
Classically, using Kolmogorov's continuity criterium, for any $\alpha \in ] 0 , 1/2 [$, defining the compact subset of ${\cal C}$
\begin{align*}
K_R = \left\{ x \in {\cal C}, \; \; |x_0| \le R, \; \sup_{t,s \in [0,T] } |x_t - x_s | \le R |t-s|^\alpha \right\},
\end{align*}
we get $ \sup_N I(\mu^N) (K_R^c) \xrightarrow[R \to \infty]{} 0$ the sequence $(I(\mu^N))_N$ is tight.

\item We have
\begin{align*}
\int_{ \| x \|_\infty \ge R} \| x \|_\infty^2 dI(\mu^N)(x) = \E \Big[ \int_{\|x\|_\infty \ge R} \| x \|_\infty^2 d\mu^N(x) \Big]
\le \frac{1}{R^{\delta}} \E \Big[ \int_x \| x \|_\infty^{2+\delta} d\mu^N(x) \Big]
\end{align*}
and taking the mean over $i$ in Proposition \ref{chap2-moments},
$$
\E \Big[ \int_x \| x \|_\infty^{2+\delta} d\mu^N(x) \Big] = 
\frac{1}{N} \sum_{i=1}^N \E \Big[ \sup_{t \in [0,T]} | X^{i,N}_t |^{2+\delta} \Big] \lesssim 1 + \int |x|^{2+\delta} d\mu^N_0(x)
$$
hence $\displaystyle{ \limsup_N \int_{\| x \|_\infty \ge R} \| x \|_\infty^2 dI(\mu^N)(x) \xrightarrow[R \to \infty]{} 0 }$.

\end{enumerate}

We can therefore introduce a subsequence and some probability space $(\widetilde \Omega, \widetilde{\Proba})$ with random variables $\widetilde \mu^{N'}, \widetilde \mu : \widetilde  \Omega \to {\cal P}_2({\cal C})$ such that
\begin{align*}
& \forall {N'}, \; \; \; \widetilde \mu^{N'} \sim \mu^{N'} \text{ in law} \\
& W_2[\widetilde \mu^{N'}, \widetilde \mu ]  \xrightarrow[N' \to \infty]{} 0 \text{ a.s } (\text{and in } L^2(\widetilde \Omega) ).
\end{align*}

Now, considering the law of the process $(\mu_t^N)_{t \ge 0}$ only, equation \myref{this} can be translated as: ~~\\
 for all $\psi \in C^2(\mathbb{R}^d)$ with $|\nabla \psi|, |\nabla^2 \psi| \lesssim1$,
\begin{align}
& M_\psi^N(t) = \langle \psi , \mu_t^N \rangle - \langle \psi , \mu_0^N \rangle - \int_0^t \langle L[\mu^N_s] \psi , \mu^N_s \rangle ds , \; \; t \ge 0, \nonumber
\\
& \text{ where } \; L[\mu] \psi =  \Big( B[\mu] + S[\mu] \Big) \cdot \nabla \psi + {\cal A}[\mu] \psi,  \label{L}
\end{align}
is a  continuous $L^2$ martingale on $\Omega$ with respect to the canonical filtration of $(\mu_t^N)_t$, whose quadratic variation is given by
\begin{align*}
\Big[ M_\psi^N \Big](t) & =  \int_0^t  \Big| \langle C[\mu_s^N] \cdot \nabla \psi  , \mu^N_s \rangle \Big| ^2 ds
+ \frac{1}{N^2} \sum_i \int_0^t \Big| \sigma^T(X_s^{i,N}) \nabla \psi(X_s^{i,N})  \Big|^2 ds 
\\
& =  \int_0^t  \Big| \langle C[\mu_s^N] \cdot \nabla \psi  , \mu^N_s \rangle \Big| ^2 ds + 
\frac{1}{N} \int_0^t \langle |\sigma^T \nabla \psi |^2  , \mu_s^N \rangle ds .
\end{align*}
Equivalently, this can be expressed as
\begin{align}
& \E   \Big[ \Big( M_\psi^N(t) - M_\psi^N(s) \Big) h(\mu^N_{t_1}, ... , \mu^N_{t_m}) \Big] = 0
\label{chap2-martingale}
\\
& \E   \Big[ \Big| M_\psi^N(t) - M_\psi^N(s) \Big|^2 h(\mu^N_{t_1}, ... , \mu^N_{t_m}) \Big]
=
\E \Big[  \Big( \Big[ M_\psi^N \Big](t) - \Big[ M_\psi^N \Big](s) \Big) \; h(\mu^N_{t_1}, ... , \mu^N_{t_m})  \Big]
\label{chap2-quadratic}
\end{align}
for all $0 \le t_1, ... , t_m \le s \le t$ and $h : {\cal P}_2(\mathbb{R}^d) \to \mathbb{R}$ continuous bounded.
Since $\widetilde \mu^{N'} \sim \mu^{N'}$ in law in ${\cal P}_2({\cal C})$, it is clear that the processes 
$(\mu_t^{N'})_{t \in [0,T]}$ and $(\widetilde \mu_t^{N'})_{t \in [0,T]}$ have the same law. It follows that \myref{chap2-martingale} and \myref{chap2-quadratic} also hold on the probability space $(\widetilde \Omega, \widetilde{\Proba})$ for
\begin{align*}
 \widetilde M_\psi^{N'}(t) = \langle \psi , \widetilde \mu_t^{N'} \rangle - \langle \psi , \mu_0^{N'} \rangle - \int_0^t \langle L[\widetilde \mu^{N'}_s] \psi , \widetilde \mu_s^{N'} \rangle ds , \; \; t \ge 0
\end{align*}
making it a continuous $L^2$ martingale on $\widetilde \Omega$, with respect to the canonical filtration of $(\widetilde \mu_t^{N'})_t$, with quadratic variation 
$\Big[ \widetilde M_\psi^{N'} \Big](t) =  \int_0^t  \Big| \langle C[\widetilde \mu_s^{N'}] \cdot \nabla \psi  , \widetilde \mu^{N'}_s \rangle \Big| ^2 ds
+ \frac{1}{N'} \int_0^t \langle |\sigma^T \nabla \psi |^2  , \widetilde \mu_s^{N'} \rangle ds 
 $.
\vspace{3mm}

We can now establish the following result.
\begin{proposition} \label{prop_mart}
For all $\psi \in C^2(\mathbb{R}^d)$ such that $| \nabla \psi |, | \nabla^2 \psi | \lesssim 1$,
\begin{align*} 
 \widetilde M_\psi(t) = \langle \psi , \widetilde \mu_t \rangle - \langle \psi , \mu_0 \rangle - \int_0^t \langle L[\widetilde \mu_s] \psi , \widetilde \mu_s \rangle ds , \; \; t \ge 0
\end{align*}
is a continuous $L^2$ martingale on $\widetilde \Omega$ with respect to the canonical filtration of $(\widetilde \mu_t)_t$, whose quadratic variation is given by
\begin{align*}
\Big[ \widetilde M_\psi \Big] (t) = \int_0^t  \Big| \langle C[\widetilde \mu_s] \cdot \nabla \psi  , \mu_s \rangle \Big| ^2 ds 
 =: V_\psi(t).
\end{align*}
\end{proposition}

\begin{proof}[Proof]

Let us work on $\widetilde \Omega$, but drop the tildas on $\mu$, $\Proba$ and the primes on $N$ for clarity. ~~\\
 Given $0 \le s \le t \le T$ and $0 \le t_1 \le ... \le t_m \le s$ and $h$ continuous bounded, we wish to send $N \to \infty$ in \myref{chap2-martingale} and \myref{chap2-quadratic}. It is enough to verify the following points:
\begin{enumerate}
\item $M_\psi^N (t) \to M_\psi(t)$ in probability
\item $|M^N_\psi(t)|^2$ is uniformly integrable in $N$
\item $\Big[ M_\psi^N \Big](t) \to V_\psi(t)$ in probability
\item$\Big[ M_\psi^N \Big] (t)$ is uniformly integrable in $N$.
\end{enumerate}
Given that $W_2[\mu^N, \mu] \to 0$ almost surely, using the immediate inequality
$$W^2_2[\mu \otimes \mu', \nu \otimes \nu' ] \le W^2_2[\mu, \nu] + W^2_2[\mu', \nu'],$$
we also derive that $W_2[(\mu^N)^{\otimes2}, \mu^{\otimes2}], W_2[(\mu^N)^{\otimes3}, \mu^{\otimes3}] \to 0$. Recalling \myref{L}, let us review the different terms involved in $M^N_\psi(t)$.
\vspace{3mm}

Since $|\psi(x)| \lesssim 1 + |x|$, we deduce from Proposition \ref{criteria} that $\langle \psi, \mu^N_t \rangle \to \langle \psi, \mu \rangle$ a.s. Moreover, $| \langle \psi, \mu^N_t \rangle | \lesssim 1 + \int \|x\|_\infty d\mu^N(x)$.
\vspace{3mm}

The term $\int_0^t \langle B[\mu_s^N] \cdot \nabla \psi , \mu_s^N \rangle ds$ can be written as
\begin{align*}
\int_0^t  \int_{(\mathbb{R}^{d})^2}  b(x,y) \cdot \nabla \psi(x) d(\mu_s^N \otimes \mu_s^N)(x,y)  ds
= \int_{ {\cal C}^2} \Big( \int_0^t  b(x_s,y_s) \cdot \nabla \psi(x_s) ds \Big) d(\mu^N \otimes \mu^N)(x,y) 
 \end{align*}
which converges almost surely to the expected term since the functional is indeed sub-quadratic ( $ | b(x,y) | \lesssim |x| + |y|$ and $| \nabla \psi | \lesssim 1$). 
The term involving $S[\mu_t^N]$ is treated in the same way.
\vspace{3mm}

 Recalling the form \myref{calA}, the term $\int_0^t \langle {\cal A}[\mu_s^N] \psi, \mu_s^N \rangle ds$ can be written similarly as
\begin{align*}
\sum_{i,j} \int_0^t \int_x \int_y \int_z \Big[ \sum_k \sigma_{i,k}(x) \sigma_{j,k}(x) + c_i(x,y) c_j(x,z) \Big] \partial^2_{i,j} \psi(x)
d(\mu_s^N \otimes \mu_s^N \otimes \mu_s^N)(x,y,z) ds .
\end{align*}
again, this converges almost surely since 
\begin{align*}
& \Big| \sigma_{i,k}(x) \sigma_{j,k}(x) \Big| \lesssim 1 + |x|^2, 
& \Big| c_i(x,y) c_j(x,z) \Big| \lesssim 1 + |x|^2 + |y|^2 + |z|^2 ,
& &  | \nabla^2 \psi | \lesssim 1 .
\end{align*}

Point $1.$ is hence proven. Let us skip Point $2.$ for now and consider $\Big[ M_\psi^N \Big](t)$.
Using the same arguments as before, 
$
\Big| \langle C[ \mu^N_s] \cdot \nabla \psi, \mu^N_s \rangle \Big|^2 
\to
\Big| \langle C[ \mu_s] \cdot \nabla \psi, \mu_s \rangle \Big|^2 
$
almost surely for fixed $s~\in~[0,t]$. The bound 
$$\Big| \langle C[ \mu^N_s] \cdot \nabla \psi, \mu^N_s \rangle \Big|^2 \lesssim 1 + \int \|x\|^2_\infty d\mu^N(x),
\;  \text{  with  }  \;
 \sup_N \E \Big[ \int \|x\|_\infty^{2 + \delta} d\mu^N(x) \Big]<\infty$$
guarantees the uniform integrability in $(t,\omega)$, so that
\begin{align*}
\int_0^t \Big| \langle C[ \mu^N_s] \cdot \nabla \psi, \mu^N_s \rangle \Big|^2 \to 
\int_0^t \Big| \langle C[ \mu_s] \cdot \nabla \psi, \mu_s \rangle \Big|^2
= V_\psi(t)
 \text{ in $L^1(\widetilde \Omega)$,}
\end{align*}
in particular in probability. Additionally, $\Big| \langle |\sigma^T \nabla \psi|^2 , \mu_s^N \rangle \Big| \lesssim 1 + \int \|x\|^2_\infty d\mu^N(x)$ which is bounded in $L^1(\widetilde \Omega)$ uniformly in $N$, so that
\begin{align*}
\frac{1}{N} \int_0^t  \langle |\sigma^T \nabla \psi|^2 , \mu_s^N \rangle \to 0 \text{ in $L^1(\widetilde \Omega)$},
\end{align*}
in particular in probability, which proves point $3$. ~~\\ 
We have in fact just seen  that $\Big[ M_\psi^N \Big](t)^{1 + \delta/2} \lesssim 1 + \int \|x\|^{2 + \delta}_\infty d\mu^N(x)$ which is bounded in $L^1(\widetilde \Omega)$ uniformly in $N$, hence giving point $4$. Finally, Burkholder-Davis-Gundy's inequality gives
\begin{align*}
\E \Big| M^N_\psi(t) \Big|^{2 + \delta} \lesssim \E \Big| \Big[ M^N_\psi \Big] (t) \Big|^{1 + \delta/2}
\end{align*}
and we derive point $2.$ from point $4$.
\end{proof}
From $(\tilde \mu_t)_{t \ge 0}$ satisfying this martingale problem stated in Proposition \ref{prop_mart}, we classically construct a weak solution using a martingale representation theorem in some Hilbert space. 

We start by noting that ${\cal P}_2(\mathbb{R}^d)$ is continuously embedded in the Sobolev space $H^{-\gamma} = \left( H^\gamma \right)'$ (where $H^\gamma= W^{\gamma,2}(\mathbb{R}^d)$) as soon as $\gamma > 1 + d/2$. Indeed, for $\psi \in H^\gamma$, 
\begin{align*}
\Big| \langle \psi, \mu  \rangle - \langle \psi, \nu \rangle \Big| \le \| \nabla \psi \|_\infty W_2[\mu, \nu]
\le C \| \psi \|_{H^\gamma} W_2[\mu,\nu]
\end{align*}
where we have used the continuous Sobolev embedding $H^\gamma \subset C^1_b$ for $\gamma > 1 + d/2$. We may consider the $H^{-\gamma-2}$-valued process
\begin{align*}
\widetilde M(t) = \widetilde \mu_t - \mu_0 - \int_0^t L[\widetilde \mu_s]^* \widetilde \mu_s ds, \; \; t \in [0,T]
\end{align*}
which satisfies, for all $\psi \in H^{\gamma+2}$ (a Sobolev embedding gives $|\nabla \psi|, |\nabla^2 \psi| \lesssim 1$),
\begin{align*}
\langle \widetilde M(t) , \psi \rangle = \widetilde M_\psi(t), \; \; t \in [0,T],
\end{align*}
which is a continuous $L^2$ martingale with respect fo the filtration 
$$
{\widetilde {\cal F}}_t = \sigma \left( \widetilde \mu_s \in {\cal P}_2(\mathbb{R}^d), \; s \in [0,t] \right),
\; \; \; t \in [0,T],
$$
 with quadratic variation $V_\psi(t)$. Using a polarisation formula, we deduce more precisely that for $\psi_1, \psi_2 \in H^{\gamma+2}$,
\begin{align}
\langle \widetilde M(t) , \psi_1 \rangle  \langle \widetilde M(t) , \psi_2 \rangle - \langle V(t) \psi_1, \psi_2 \rangle, \; \; t \in [0,T]
\end{align}
is a continuous $({\widetilde {\cal F}}_t)_t$-martingale, where 
$ \langle V(t) \psi_1, \psi_2 \rangle = \int_0^t \langle C[\widetilde \mu_s] \cdot \nabla \psi_1 , \widetilde \mu_s \rangle \langle C[\widetilde \mu_s] \cdot \nabla \psi_2 , \widetilde \mu_s \rangle ds$. The martingale representation theorem from \cite{da_prato} p222 (Theorem 9.2) then holds, giving another probability space $(\widehat \Omega, \widehat {\cal F}, \widehat{\Proba})$ with a filtration $(\widehat{\cal F}_t)_{t \in [0,T]}$ and a $(\widetilde{\cal F}_t \times \widehat{\cal F}_t)$-brownian motion $(W_t)_{t \in [0,T]}$ on 
$(\widetilde \Omega \times \widehat \Omega, \widetilde{\Proba} \otimes \widehat{\Proba})$ 
such that
\begin{align}
\widetilde M(t)(\tilde \omega, \hat \omega) := \widetilde M(t)(\tilde \omega)
= - \int_0^t \nabla \cdot \left( C[\widetilde \mu_s(\tilde \omega)] \widetilde \mu_s(\tilde \omega) \right) dW_s(\tilde \omega, \hat \omega).
\end{align}
It follows that $(\tilde \omega, \hat \omega) \mapsto (\widetilde \mu_t(\tilde \omega))_{t \in [0,T]}$ is a solution of \myref{chap2-spde} on $\widetilde \Omega \times \widehat \Omega$ according to Definition~\ref{solution} (whose law is of course the same as that of $\tilde \omega \mapsto (\widetilde \mu_t(\tilde \omega))_{t \in [0,T]}$).

\end{subsection}

\end{section}


\begin{section}{Strong mean-field convergence}

In this section, for simplicity, we restrict ourselves to the setting of a common noise, according to
Assumption \ref{hyp_common}. 

In this case, the limiting SPDE \myref{chap2-spde} becomes a stochastic conservation equation (given by \myref{conservation}) and solutions $\mu_t$ are naturally expected to be obtained as the push-forward measures of $\mu_0$ through the flow of the associated (non-linear) stochastic characteristics.

\begin{subsection}{Stochastic characteristics.} \label{section_chara}

Let us suppose that Assumptions \ref{hyp1} and \ref{hyp2} hold.

\begin{define} \label{transport_form}~~\\
Given some random $\mu  \in {\cal P}_2({\cal C})$ such that $\E \Big[ \int \| x \|^2_\infty d\mu(x) \Big] < \infty$, the characteristics $X^\mu$ are defined as the solution of
\begin{equation}
\left\{
\begin{array}{l}
 dX_t^\mu(x) = \Big( B[\mu_t] + S[\mu_t] \Big)(X^\mu_t(x)) dt + C[\mu_t](X^\mu_t(x)) d\beta_t, \; \; t \in [0,T],
\\
 X_0^\mu(x) = x \in \mathbb{R}^d .
 \end{array}
 \right.
\label{chap2-chara}
\end{equation}
A random measure $\mu : \Omega \to  {\cal P}_2({\cal C})$ is said to be "of the transport form" if it satisfies the fixed-point like identity
\begin{align}
\mu = (X^\mu)^* \mu_0 \; \; a.s 
\label{the_transport}
\end{align}
where $X^\mu = (X_t^\mu(x))_{t \in [0,T], x \in \mathbb{R}^d}$ is the flow of characteristics associated to \myref{chap2-chara} and the measure $(X_t^\mu)^*\mu_0 \in {\cal P}({\cal C})$ is defined by: for all 
$m \ge 1, \; \; t_1, \ldots t_m \in [0,T], \; \; \psi \in C_b((\mathbb{R}^d)^m)$,
\begin{align*}
\int_{{\cal C}} \psi(x_{t_1}, \ldots, x_{t_m}) d \left( (X^\mu)^* \mu_0 \right)(x) = 
\int_{\mathbb{R}^d} \psi(X^\mu_{t_1}(x), \ldots X^\mu_{t_m}(x)) d\mu_0(x) .
\end{align*}
\end{define}

\begin{remark} \label{chap2-continuity}

Using Assumptions \ref{hyp1} and \ref{hyp2}, given $\mu  \in {\cal P}_2({\cal C})$ satisfying $\E \Big[ \int \| x \|^2_\infty d\mu(x) \Big] < \infty$, one may easily establish that, for any fixed $x \in \mathbb{R}^d$, $\mathbb{E} \Big[ \sup_{t \in [0,T]} \Big| X^\mu_t(x) \Big|^2 \Big] < \infty$, so that \myref{chap2-chara} admits a unique global solution.

Moreover, it is to be noted that the flow $x \mapsto X^\mu(x) \in {\cal C}$ is almost-surely continuous, so that the push-forward measure $(X^\mu)^*\mu_0$ is indeed well defined. This can easily be seen in the case where the kernels $b$, $c$, $s_1$ are globally Lipschitz-continuous, since we can derive some Kolmogorov estimate of the form
\begin{align*}
\E \Big[ \sup_{t \in [0,T]} | X^\mu_t(x) - X^\mu_t(x') |^p \Big] \lesssim |x-x'|^p .
\end{align*}
The result follows in the locally Lipschitz-continuous case using a classic stopping-time method (found e.g in \cite{fang}).
\end{remark}

\begin{remark}
For some fixed $N \ge 1$, let $(X^{i,N})^{i=1, \ldots, N}$ be the solution of the particle system \myref{model} with initial data $(x_0^{i,N})^{i=1, \ldots, N}$ and let $\mu^N= \frac{1}{N} \sum_{i=1}^N \delta_{X^{i,N}}$ the associated empirical measure.~~\\
Then one can see that $X_t^{i,N} = X_t^{\mu^N}(x_0^{i, N})$ so that $\mu^N = (X^{\mu^N})^* \mu_0^N$ is of the transport form.
\end{remark}

Measures of the transport form are, by design, solutions of the conservation equation \myref{conservation}

\begin{proposition} \label{transport_solution}
Let $\mu_0 \in {\cal P}_2({\cal C})$ and $\mu = (X^\mu)^* \mu_0$ be of the transport form.~~\\
Then $\mu = (\mu_t)_{t \in [0,T]}$ satisfies the SPDE \myref{conservation} in the sense of Definition \ref{solution}.
\end{proposition}
\begin{proof}[Proof]
Firstly, using the same reasoning as in the proof of Proposition \ref{chap2-moments}, one can easily show
\begin{align}
\E \int_{\cal C} \| x \|^2_\infty d\mu(x) = \E \int_{\mathbb{R}^d} \sup_{t \in [0,T]} \Big| X^\mu_t(x) \Big|^2 d\mu_0(x)
\lesssim 1 + \int_{\mathbb{R}^d} |x|^2 d\mu_0(x),
\label{estimate_mu}
\end{align}
so that the characteristics \myref{chap2-chara} are globally well-defined. For any $\psi \in C^2_c(\mathbb{R}^d)$, since $\mu \in {\cal P}({\cal C})$, the process $(\langle \psi, \mu_t \rangle)_{t \in [0,T]}$ is automatically (adapted and) almost surely continuous. Itô's formula then results in
\begin{align*}
\psi(X^\mu_t(x)) & = \psi(x) + \int_0^t \nabla \psi(X^\mu_s(x)) \cdot \left( B[\mu_s] + S[\mu_s] + {\cal A}[\mu_s] \right)(X^\mu_s(x)) ds 
\\ & \; \; \; \; + \int_0^t \nabla \psi(X^\mu_s(x)) \cdot C[\mu_t](X^\mu_s(x)) d\beta_s .
\end{align*}
Note that, using the sublinearity Assumption \ref{hyp1} and the estimate \myref{estimate_mu}, the stochastic integral involved here easily defines a square-integrable martingale.
Integrating with respect to $d\mu_0(x)$ using a stochastic Fubini theorem gives exactly \myref{eq_psi}, so that Definition \ref{solution} is satisfied.
\end{proof}

We now formulate an estimate which locally compares two solutions of the transport form.

\begin{proposition}[${\cal P}_p$-comparison estimate for compactly-supported initial data] ~~\\
\label{comparison}
Let $\mu_0, \widetilde \mu_0 \in {\cal P}(\mathbb{R}^d)$  be supported in some compact set $K \subset \mathbb{R}^d$, and let $\mu, \widetilde \mu : \Omega \to {\cal P}_2({\cal C})$ be of the transport form.  For all $p > 1$, $R > 0$, defining the stopping time
\begin{align*}
\tau_R = \inf \left\{ t \ge 0, \; \sup_{x \in K} \Big( |X^\mu_t(x)| + |X^{\widetilde \mu}_t(x)| \Big) \ge R  \right\} \wedge T
\end{align*}
there exists some constant $C_{p,R,T} > 0$ such that
\begin{align}
\E \Big[ \sup_{t \in [0, \tau_R]} W_p^p[\mu_t, \widetilde \mu_t] \Big] \le C_{p,R,T} W^p_p[\mu_0, \widetilde \mu_0] .
\end{align}
\end{proposition}
\begin{remark}
Given some sequence $(x_k)_{k \ge 1}$ dense in $K$, the continuity of $x \mapsto X_0^\mu(x) \in {\cal C}$ gives
\begin{align*}
& \left\{ \tau_R \ge t \right\} = 
\left\{  \sup_{x \in K} \sup_{[0,t]}  \Big( |X^\mu_s(x) | + |X^{\widetilde\mu}_s(x) | \Big) \le R \right\}
= \bigcap_{k \ge 1} \left\{ \tau_R(x_k) \ge t \right\}
\end{align*}
where $\tau_R(x) = \inf \left\{ t \ge 0, \;  |X^\mu_t(x)| + |X^{\widetilde \mu}_t(x)|  \ge R  \right\} \wedge T$ is a stopping time, so that $\tau_R$ is indeed a stopping time.
\end{remark}

\begin{remark} \label{work}
The proof of this comparison estimate relies strongly on measures of the transport form \myref{the_transport} which are natural solutions of the stochastic conservation equation \myref{conservation}. Whenever $\sigma \neq 0$, solutions of SPDE \myref{chap2-spde} no longer exhibit a natural "transport form". Moreover, for fixed $N \ge 1$, the empirical measure $\mu^N$ cannot be written as the solution of some SPDE (see \myref{this}). The case of a particle system with independent noise therefore requires additional work.
\end{remark}

\begin{proof}[Proof of Proposition \ref{comparison}]
For the sake of making calculations clearer, we only treat the case $p=2$.
Let us forget about $S[\mu]$ since it plays the same role as $B[\mu]$. Let us introduce a local Lipschitz constant $c_R > 0$ so that for all $|x|, |y|, |x'|, |y'| \le R$,
\begin{align*}
& |b(x,y) - b(x',y')| \le c_R \Big( |x-x'| + |y-y'| \Big) 
\\
&  |c(x,y) - c(x',y')| \le c_R \Big( |x-x'| + |y-y'| \Big) 
\\
&|c(x,y)| \le c_R.
\end{align*}
This easily results in the following for all $|x|, |x'| \le R$, $\nu$, $\nu'$ with support in $B(0,R)$,
\begin{equation}
\begin{array}{l}
 \Big| B[\nu](x) - B[\nu'](x') \Big| 
 \le c_R \Big( |x-x'| + W_1[\nu,\nu'] \Big) \le c_R \Big( |x-x'| + W_2[\nu,\nu'] \Big)
 \\
  \Big| C[\nu](x) - C[\nu'](x') \Big| 
 \le c_R \Big( |x-x'| + W_1[\nu,\nu'] \Big) \le c_R \Big( |x-x'| + W_2[\nu,\nu'] \Big)
 \\
 \Big| C[\nu](x) \Big| \le c_R .
 \end{array}
\label{support_est}
\end{equation}

Using Theorem 4.1 from \cite{villani}, we may introduce an optimal plan $\pi \in \Pi(\mu_0, \widetilde \mu_0)$ so that
\begin{align*}
W_2^2[\mu_0, \widetilde \mu_0] = \int_{K^2} |x-y|^2 d\pi(x,y).
\end{align*}
Since $\mu$, $\widetilde \mu$ are of the transport form, denoting $X_t = X_t^\mu$ and $\widetilde X_t = X_t^{\widetilde \mu}$, introducing the mapping $T : (x,y) \in (\mathbb{R}^d)^2 \mapsto (X_t(x), \widetilde X_t(y)) \in (\mathbb{R}^d)^2$, one can easily see that $T^* \pi \in \Pi(\mu_t, \widetilde \mu_t)$. It follows that
\begin{align}
W_2^2[\mu_t, \widetilde \mu_t] \le J_t:= \int_{K^2} \Big| X_t(x) - \widetilde X_t(y) \Big|^2 d\pi(x,y) .
\end{align}
We now apply Itô's formula to $\eta_t(x,y) = X_t(x) - \widetilde X_t(y)$ to get
\begin{align*}
d | \eta_t(x,y) |^2 & = \left( 2 \eta_t(x,y) \cdot \Big[ B[\mu_t](X_t(x)) - B[\widetilde \mu_t](\widetilde X_t(y)) \Big] + \Big| C[\mu_t](X_t(x)) - C[\widetilde \mu_t] (\widetilde X_t(y)) \Big|^2 \right) dt
\\
& \; \; \; + 2 \eta_t(x,y) \cdot  \Big( C[\mu_t](X_t(x)) - C[\widetilde \mu_t](\widetilde X_t(y)) \Big) d\beta_t .
\end{align*}
Applying \myref{support_est}, we deduce, for some $C_R > 0$,
\begin{align}
d | \eta_{t \wedge \tau_R} (x,y) |^2 & \le C_R \Big( |\eta_{t \wedge \tau_R} (x,y) |^2 + W_2^2[\mu_{t \wedge \tau_R}, \widetilde \mu_{t \wedge \tau_R} ] \Big) dt + dM_{t \wedge \tau_R}(x,y) \nonumber \\
& \le C_R \Big(  |\eta_{t \wedge \tau_R} (x,y) |^2 + J_{t \wedge \tau_R} \Big) dt + dM_{t \wedge \tau_R} (x,y)
\label{bdgg}
\end{align}
with 
$\displaystyle{ M_t(x,y) = \int_0^t 2 \eta_s(x,y) \cdot  \Big( C[\mu_s](X_s(x)) - C[\widetilde \mu_s](\widetilde X_s(y)) \Big) d\beta_s }$. We may integrate this expression with respect do $d\pi(x,y)$ using a stochastic Fubini theorem to get
\begin{align}
d J_{t \wedge \tau_R} \le C_R J_{t \wedge \tau_R} dt + dM_{t \wedge \tau_R}
\label{J_expr}
\end{align}
with $\displaystyle{ M_t = \int_0^t \left( \int_{K^2} 2 \eta_s(x,y) \cdot  \Big( C[\mu_s](X_s(x)) - C[\widetilde \mu_s](\widetilde X_s(y)) \Big) d\pi(x,y) \right) dt } $ . Taking the expectation in \myref{J_expr} and applying Grönwall's lemma leads to
\begin{align*}
\forall t \in [0,T], \; \; \; \E[ J_{t \wedge \tau_R} ] \le C_{R,T} W_2^2[\mu_0,\widetilde \mu_0].
\end{align*}
Coming back to \myref{J_expr} one may now write, using Burkholder-Davis-Gundy's inequality
\begin{align*}
\E \Big[ \sup_{[0,T]} J_{t \wedge \tau_R} \Big] \le C_R \int_0^T \E [ J_{t \wedge \tau_R} ] dt + \E[ \sup_{[0,T]} M_{ t \wedge \tau_R} ] \le C_{R,T} W_2^2[\mu_0, \widetilde \mu_0] + C \E\Big( [M]_{T \wedge \tau_R}^{1/2} \Big) .
\end{align*}
With
\begin{align*}
[M]_{T \wedge \tau_R} & = 4 \int_0^T \Big| \int_{K^2}  \eta_{t \wedge \tau_R} (x,y) \cdot  \Big( C[\mu_{t \wedge \tau_R}](X_{t \wedge \tau_R}(x)) - C[\widetilde \mu_{t \wedge \tau_R}](\widetilde X_{t \wedge \tau_R}(y)) \Big) d\pi(x,y)   \Big|^2 dt \\
& \le 4 c_R \int_0^T \left( \int_{K^2} | \eta_{t \wedge \tau_R}(x,y)| \Big( \eta_{t \wedge \tau_R}(x,y) + W_2[\mu_{t \wedge \tau_R}, \widetilde \mu_{t \wedge \tau_R}] \Big) d\pi(x,y) \right)^2  dt
\\
& \le C_R \int_0^T (J_{t \wedge \tau_R})^2 dt \le C_R \left( \sup_{[0,T]} J_{t \wedge \tau_R} \right) \int_0^T J_{t \wedge \tau_R} dt .
\end{align*}
Hölder's inequality classically gives
\begin{align*}
\E \Big[ \sup_{[0,T]} J_{t \wedge \tau_R} \Big]
\le C_{R,T} W_2^2[\mu_0, \widetilde \mu_0] + C_{R,T} \E \Big[ \sup_{[0,T]} J_{t \wedge \tau_R} \Big]^{1/2} W_2[\mu_0, \widetilde \mu_0]
\end{align*}
from which we easily derive
\begin{align}
\E \Big[ \sup_{t \in [0, \tau_R]} W_2^2[\mu_t, \widetilde \mu_t] \Big] \le \E \Big[ \sup_{[0,T]} J_{t \wedge \tau_R} \Big] \le C_{R,T} W_2^2[\mu_0, \widetilde \mu_0] .
\end{align}

\end{proof}

\begin{remark}
Seeing $\mu = (X)^* \mu_0$ and $\widetilde \mu = (\widetilde X)^* \widetilde \mu_0 $ as random elements of ${\cal P}_p({\cal C})$, we can in fact be a little more precise. With $\pi \in \Pi(\mu_0, \widetilde \mu_0)$ an optimal plan between $\mu_0$ and $\widetilde \mu_0$, we have
\begin{align}
W_p^p[\mu, \widetilde \mu] \le J^*_T:= \int_{K^2} \sup_{t \in [0,T]} \Big| X_t(x) - \widetilde X_t(y) \Big|^p d\pi(x,y)
\end{align}
and one could easily adapt the proof (apply Burkholder-Davis-Gundy's inequality in \myref{bdgg} before integrating) to get the estimate
\begin{align}
\E \left[ J^*_{\tau_R} \right] \le C_{p,R,T} W_p^p[\mu_0, \widetilde \mu_0] .
\label{ineq_J}
\end{align}
\end{remark}
~~\\
The result from Proposition \ref{comparison} makes it clear that, given a compactly-supported measure $\mu_0$, one should naturally require some estimates regarding the growth of the support of $\mu_t$, that is, estimates on $\sup_{x \in K} |X^\mu_t(x)|$. 
In \cite{choi-salem} and \cite{jung-ha} for instance, where the diffusion coefficient $c(x,y)$ is linear, precise almost-sure estimates for the support of $\mu_t$ are achieved using some stochastic Grönwall inequality. 

The assumptions from Theorem \ref{chap2-thm2} provide another setting (where, in particular, the diffusion coefficient is bounded) in which we are able to obtain a bound on the moments 
$$\E \left[ \sup_{x \in K} \sup_{[0,T]} |X^\mu_t(x)|^p \right], \; \; \; p \ge 1.$$

\end{subsection}

\begin{subsection}{Properties of the coefficients.} \label{detailed_assumptions}

From this point on, we suppose that Assumptions \ref{hyp1'} and \ref{hyp2'} are satisfied.
\vspace{3mm}

Note that, as mentioned in the introduction, these assumptions allow us to consider stochastic Cucker-Smale models with "truncated velocities" in the interaction perturbation, given by \myref{pert4}. Indeed, this corresponds to coefficients of the form
\begin{equation}
b((x,v);(y,w)) = \left(
\begin{array}{c}
v \\
\psi(x-y)(w-v)
\end{array}
\right)
\; \; \; \;
c((x,v) ; (y,w) ) =
\left(
\begin{array}{c}
0 \\
\phi(x-y){\cal R}(w-v)
\end{array}
\right) .
\end{equation}
Provided that the weight functions $\psi, \phi$ and the truncation function $\cal R$ satisfy, for some $\theta \in [0,1)$,
\begin{align*}
& |\psi(x)|, \; \;  |\phi(x)|, \; \;  |{\cal R}(v)| \lesssim 1,
\\
& |\psi(x) - \psi(y) | \lesssim (1 + |x|^\theta + |y|^\theta) |x-y|,
\\
& |\phi(x) - \phi(y) | \lesssim (1 + |x|^\theta + |y|^\theta) |x-y|,
\\
& |{\cal R}(v) - {\cal R}(w) | \lesssim (1 + |v|^\theta + |w|^\theta) |v-w|,
\\
& | \nabla {\cal R}(v) - \nabla {\cal R}(w) | \lesssim (1 + |v|^{2 \theta} + |w|^{2\theta} ) |v-w|,
\end{align*}
one can check that all the required assumptions are satisfied, with, denoting $z_i = (x_i,v_i)$,
\begin{equation*}
s_1(z_1,z_2,z_2) = \left(
\begin{array}{c}
0
\\
- \phi(x_1-x_2) \phi(x_1-x_3) \nabla {\cal R}(v_2-v_1) {\cal R}(v_3-v_1)  
\\
\hspace{20mm} + \phi(x_1-x_2) \phi(x_2-x_3) \nabla {\cal R}(v_2-v_1) {\cal R}(v_3-v_2)
\end{array}
\right) .
\end{equation*}

From Assumptions \ref{hyp1'} and \ref{hyp2'}, we easily derive
\begin{align}
& \Big| B[\mu](x) | \lesssim 1 + |x| + \int |y| d\mu(y)
\label{sub_B}
\\
& \Big| B[\mu](x) - B[\mu](x') \Big| \lesssim \Big( 1 + |x|^{2\theta} + |x'|^{2\theta} + \int |y|^{2 \theta} d\mu(y) \Big) |x-x'|  
\label{lip_B}
\end{align}
and similar estimates for $S[\mu]$, as well as
\begin{align}
& \Big| C[\mu](x) \Big| \lesssim 1 
\label{sub_C}
\\
& \Big| C[\mu](x) - C[\mu](x') \Big| \lesssim \Big( 1 + |x|^\theta + |x'|^\theta + \int |y|^\theta d\mu(y) \Big) |x-x'|  .
\label{lip_C}
\end{align}

\end{subsection}

\begin{subsection}{Estimates for the stochastic characteristics.} \label{section_chara2}

In this context, let us start by establishing some exponential moments for the stochastic ~~\\ characteristics.
\begin{lemma}[Exponential moments]  \label{prop_exp}~~\\
Let $\mu_0 \in {\cal P}_2({\mathbb{R}^d})$ and $\mu : \Omega \to {\cal P}_2({\cal C})$ be of the transport form. ~~\\
Then, for all $T > 0$, $\alpha_0 \in (0,1]$, there exists $\alpha_T > 0$  and $C_T > 0$ such that for $t \in [0,T]$,
$x \in \mathbb{R}^d$,
\begin{align}
 \E \left[ 
\exp\left(\alpha_T \left( |X^\mu_t(x)|^2 + \int |y|^2 d\mu_t(y) \right) \right) \right] \le
 C_T \exp \left( \alpha_0 \left( |x|^2 + \int |y|^2 d\mu_0(y) \right) \right) .
\end{align}
\end{lemma}
\begin{proof}[Proof]
This method is inspired from the one developed in \cite{bolley}, Lemma 3.5.
Again, let us forget $S[\mu]$ since it satisfies the same estimates as $B[\mu]$. Itô's formula gives
\begin{align*}
& d \Big[ |X^\mu_t(x)|^2 \Big]  = \left( 2 X_t^{\mu}(x) \cdot B[\mu_t] (X^\mu_t(x)) + |C[\mu_t](X^\mu_t(x))|^2 \right) dt + 2 X_t^\mu(x) \cdot C[\mu_t](X^\mu_t(x))  d\beta_t   .
\end{align*}
Integrating with respect to $d\mu_0(x)$ leads to
\begin{align*}
d \Big[ \int |y|^2 d\mu_t(y) \Big] =   \Big[ \int \Big(2 y \cdot B[\mu_t](y) + |C[\mu_t](y)|^2  \Big)  d\mu_t(y) \Big] dt
+  \Big[ \int \Big(2 y \cdot C[\mu_t](y) \Big) d\mu_t(y) \Big] d\beta_t .
\end{align*}
Hence, letting $Y_t = |X^\mu_t(x)|^2 + \int |y|^2 d\mu_t$ and summing these two identities, we get
\begin{align*}
 dY_t = a_t dt + \sigma_t d\beta_t
 \end{align*}
 with, using \myref{sub_B} and \myref{sub_C}, $| a_t |, |\sigma_t|^2 \lesssim (1 + Y_t)$.
Let $\alpha(t)$ be a deterministic, positive smooth function to be fixed later on.
Letting $Z_t = \exp\Big(\alpha(t) Y_t \Big)$, it follows that, for some $C > 0$,
\begin{align*}
d Z_t  & = Z_t \Big( \alpha'(t) Y_t + \alpha(t) a_t + \frac{\alpha(t)^2}{2} \sigma_t^2 \Big) dt + Z_t \alpha(t) 
\sigma_t d\beta_t
\\
& \le Z_t \Big(  \alpha'(t) Y_t + C(1 + Y_t) (\alpha(t) + \alpha(t)^2) \Big) dt + Z_t \alpha(t) 
\sigma_t d\beta_t \\
& = Z_t  Y_t \Big( \alpha'(t) + C \alpha(t) + C \alpha^2(t)  \Big) dt + C Z_t \Big( \alpha(t) + \alpha(t)^2 \Big) dt 
+ Z_t \alpha(t) \sigma_t d\beta_t .
\end{align*}
Choosing $\alpha(t)$ so that $\alpha' + C \alpha + C \alpha^2 \le 0$, that is for instance $\alpha(t) = \alpha_0 e^{-2Ct}$, we are led to
\begin{align}
dZ_t \le C \left( \alpha_0 + \alpha_0^2 \right) Z_t dt + Z_t \alpha(t) \sigma_t d\beta_t,
\label{Z}
\end{align}
hence taking the expectation (again, one may use a stopping time to be more rigorous) and applying Grönwall's lemma gives
\begin{align*}
\E \Big[ Z_t \Big] \le \exp \left(C ( \alpha_0 + \alpha_0^2 ) T \right)\exp \left( \alpha_0 \left( |x|^2 + \int |y|^2 d\mu_0(y) \right) \right) .
\end{align*}
This is the expected result with $\alpha_T = \inf_{t \in [0,T]} \alpha(t) = \alpha_0 e^{-2CT}$ and $C_T = \exp \left(C(\alpha_0 + \alpha_0^2) T\right)$. \end{proof}
 
We can now establish some bounds regarding the support of $\mu_t$.

\begin{proposition}[Kolmogorov continuity estimates for the stochastic characteristics] 
\label{chap2-kolmo}
~~\\
Let $\mu_0 \in {\cal P}(\mathbb{R}^d)$ be supported in some compact set $K \subset \mathbb{R}^d$, and $\mu : \Omega \to {\cal P}_2({\cal C})$ be of the transport form. For all $T > 0$, $p > 1$, there exists a constant $C_{K,T,p}$ such that
\begin{align*}
\forall x,x' \in K, \; \; \; \; \E \left[ \sup_{t \in [0,T]} | X^\mu_t(x) - X^\mu_t(x') |^p \right] \le C_{K,T,p} |x-x'|^p .
\end{align*}
\end{proposition}
Using Kolmogorov's continuity theorem, one can then bound all the moments of the $\alpha$-Hölder constant
\begin{align*}
N_\alpha(X^\mu) = \sup_{x,x' \in K} \sup_{t \in [0,T]} \frac{ |X^\mu_t(x) - X^\mu_t(x')|}{|x-x'|^\alpha}
\end{align*}
for all $\alpha \in (0,1)$. The set $K$ being compact, an immediate consequence is the following.

\begin{corollary} \label{coro}
Let $\mu_0 \in {\cal P}(\mathbb{R}^d)$ be supported in some compact set $K \subset \mathbb{R}^d$, and $\mu : \Omega \to {\cal P}_2({\cal C})$ be of the transport form. For all $T > 0$, $p > 1$, there exists a constant $C_{K,T,p}$ such that
\begin{align*}
\E \left[ \sup_{x \in K} \sup_{[0,T]} | X_t^\mu(x)|^p \right] \le C_{K,T,p} .
\end{align*}
\end{corollary}

\begin{proof}[Proof of Proposition \ref{chap2-kolmo}]
Let us once again forget $S[\mu]$. Letting $\eta_t = X^\mu_t(x) - X^\mu_t(x')$, Itô's formula gives
\begin{align*}
d | \eta_t |^{2p} = 2p |\eta_t|^{2p-2} \eta_t \cdot d\eta_t + 2p(p-1) \sum_{i,j=1}^d |\eta_t|^{2p-4} \eta_t^i \eta_t^j d[\eta^i, \eta^j]_t + p \sum_{i=1}^d |\eta_t|^{2p-2} d[\eta^i]_t
\end{align*}
with
\begin{align*}
& d\eta_t = \left( B[\mu_t](X^\mu_t(x)) - B[\mu_t](X^\mu_t(x')) \right) dt + \left( C[\mu_t](X^\mu_t(x)) - C[\mu_t](X^\mu_t(x')) \right) d\beta_t
\\
& d [\eta^i, \eta^j ]_t =  \left( C^i[\mu_t](X^\mu_t(x)) - C^i[\mu_t](X^\mu_t(x')) \right)  \left( C^j[\mu_t](X^\mu_t(x)) - C^j[\mu_t](X^\mu_t(x')) \right) dt .
\end{align*}
Using the Lipschitz estimates from \myref{lip_B} and \myref{lip_C}, we derive that, for some $C_* \ge 1$,
\begin{align}
& d |\eta_t|^{2p} \le \lambda_t |\eta_t|^{2p} dt + dM_t, \; \; \text{ with} \label{eta_expr}
\\
& \lambda_t \equiv \lambda_t(x,x') := C_* \left( 1 + |X^\mu_t(x)|^{2\theta} + |X^\mu_t(x')|^{2\theta} + \int |y|^{2 \theta} d\mu_t \right),  \label{lambda_expr} \\
& M_t = \int_0^t 2p |\eta_s|^{2p-2} \eta_s \cdot \left( C[\mu_s](X^\mu_s(x)) - C[\mu_s](X^\mu_s(x')) \right) d\beta_s .
\end{align}
Let us define 
\begin{align}
\Lambda_t \equiv \Lambda_t(x,x'):= \gamma \int_0^t \lambda_s(x,x') ds,
\end{align}
where the consant $\gamma \ge 1$ is to be fixed later on. One can now write
\begin{align}
 \E \left[ \sup_{[0,T]} | \eta |^p \right] & \le \E \left[ \exp\left(\frac{\Lambda_T}{2}\right)  \sup_{[0,T]}  \left( \exp \left(-\frac{\Lambda_t}{2} \right) |\eta_t|^{p} \right) \right]  \nonumber \\
& \le \E \Big[ \exp(\Lambda_T) \Big]^{1/2} \E \Big[ \sup_{[0,T]}  \left( \exp \left(-\Lambda_t \right) |\eta_t|^{2p} \right) \Big]^{1/2} .
\label{holderisation}
\end{align}
Let us fix $\alpha_0 = 1$ and introduce $\alpha_T > 0$ such that the estimate from Proposition \ref{prop_exp} holds. Then,
\begin{align}
 \E \Big[ \exp(\Lambda_T) \Big] & \le \int_0^T \E \Big[ \exp(\gamma \lambda_t ) \Big] dt \nonumber
 \\
 &
 =  \int_0^T \E \Big[ \exp \left(\gamma C_* \left( 1 + |X^\mu_t(x)|^{2\theta} + |X^\mu_t(x')|^{2\theta} + \int |y|^{2 \theta} d\mu_t  \right) \right) \Big] dt \nonumber
 \\
 & \le C_\#  \int_0^T \E \Big[ \exp \left( \frac{\alpha_T}{2} \left( |X^\mu_t(x)|^2 + |X^\mu_t(x')|^2 + \int |y|^2 d\mu_t(y) \right)  \right) \Big] dt, \label{exp_ineq}
\end{align}
where the constant $C_\# = C_\#(T,\gamma) > 0$ is chosen large enough so that  (recall that $\theta \in [0,1)$)
\begin{align*}
\forall u \in \mathbb{R}^+, \; \; \; \exp \left( \gamma C_* (1 + u^{2 \theta}) \right) \le C_\# \exp\left( \frac{\alpha_T}{2} u^2 \right) .
\end{align*} 
We may now use Hölder's inequality in \myref{exp_ineq} and apply the estimate from Proposition \ref{prop_exp} to conclude
\begin{align}
\E \Big[ \exp(\Lambda_T) \Big] \le C_{K,T} . \label{big_exp}
\end{align}
Combining \myref{holderisation} and \myref{big_exp}, it only remains to prove that
\begin{align}
 \E \Big[ \sup_{[0,T]}  \left( \exp \left(-\Lambda_t \right) |\eta_t|^{2p} \right) \Big] \le C_{K,T} |x-x'|^{2p} .
\label{remain}
\end{align}
By design, we derive from \myref{eta_expr} with Itô's formula that
\begin{align*}
\exp(-\Lambda_t) |\eta_t|^{2p} \le |x-x'|^{2p} + \int_0^t \exp(-\Lambda_s) dM_s
\end{align*}
so that, denoting $N_t$ the martingale term, with Burkholder-Davis-Gundy's inequality,
\begin{align}
\E \Big[ \sup_{[0,T]}  \left( \exp \left(-\Lambda_t \right) |\eta_t|^{2p} \right) \Big] 
\lesssim |x-x'|^{2p} + \E \left( [N]_T^{1/2} \right) .
\label{BDG_again}
\end{align}
This quadratic variation is given by
\begin{align*}
[ N ]_T & = 4p^2 \int_0^T \exp(-2 \Lambda_t) \Big|
|\eta_s|^{2p-2} \eta_s \cdot \left( C[\mu_s](X^\mu_s(x)) - C[\mu_s](X^\mu_s(x')) \right) \Big|^2 dt
\\
& \lesssim \int_0^T \exp(-2 \Lambda_t) \lambda_t |\eta_t|^{4p} dt \lesssim  
\sup_{[0,T]}  \left( \exp \left(-\Lambda_t \right) |\eta_t|^{2p} \right) \int_0^T \exp \left(-\Lambda_t \right) \lambda_t |\eta_t|^{2p} dt,
\end{align*}
so that \myref{BDG_again} leads to
\begin{align*}
& \E \Big[ \sup_{[0,T]}  \left( \exp \left(-\Lambda_t \right) |\eta_t|^{2p} \right) \Big] \\
& \hspace{5mm} \lesssim |x-x'|^{2p} 
 + \E \Big[ \sup_{[0,T]}  \left( \exp \left(-\Lambda_t \right) |\eta_t|^{2p} \right) \Big]^{1/2} \left( \int_0^T \E \Big[ \exp \left(-\Lambda_t \right) \lambda_t |\eta_t|^{2p} \Big] dt \right)^{1/2} .
\end{align*}
The estimate \myref{remain} will therefore hold if we can establish
\begin{align}
\forall t \in [0,T], \; \; \; \E \Big[ \exp \left(-\Lambda_t \right) \lambda_t |\eta_t|^{2p} \Big] \le C_{K} |x-x'|^{2p} .
\end{align}
The integration by part formula gives
\begin{align}
d \left[  \exp \left(-\Lambda_t \right) \lambda_t |\eta_t|^{2p} \right] & = \exp \left(-\Lambda_t \right) 
\left( |\eta_t|^{2p} d\lambda_t + \lambda_t d ( |\eta_t|^{2p} ) + d[\lambda, |\eta|^{2p} ]_t - \Lambda'_t \lambda_t |\eta_t|^{2p} dt \right) \nonumber
\\
& = \exp \left(-\Lambda_t \right) 
\left( |\eta_t|^{2p} d\lambda_t + \lambda_t d ( |\eta_t|^{2p} ) + d[\lambda, |\eta|^{2p} ]_t - \gamma \lambda_t^2 |\eta_t|^{2p} dt \right) .
\label{drift}
\end{align}
Given that 
$\E \left[  \exp \left(-\Lambda_0 \right) \lambda_0 |\eta_0|^{2p} \right] \le C_K |x-x'|^{2p} $, 
it is enough to prove that the drift terms in \myref{drift} are all negative for $\gamma$ chosen large enough. 
~~\\
Recalling \myref{eta_expr},  $ \lambda_t d ( |\eta_t|^{2p} ) \lesssim \lambda_t^2 |\eta_t|^{2p} dt + \lambda_t dM_t$.~~\\
Moreover, using the sublinearity Assumption~\ref{hyp2'}, we easily get from the expression \myref{lambda_expr}
\begin{align}
& d \lambda_t \lesssim \lambda_t + dm_t(x) + dm_t(x') + \int dm_t(y) d\mu_0(y),
\\
& \text{ where } dm_t(y) = 2\theta |X_t^\mu(y)|^{2 \theta - 2} X_t^\mu(y) \cdot C[\mu_t](X^\mu_t(y)) d\beta_t,
\end{align}
so that $ |\eta_t|^{2p} d\lambda_t \lesssim  \lambda_t |\eta_t|^{2p} dt + |\eta_t|^{2p} \left( dm_t(x) + dm_t(x') + \int dm_t(y) d\mu_0(y) \right)$. ~~\\
We conclude by noting that 
\begin{align*}
d[\lambda, |\eta|^{2p} ]_t = d \Big[ m(x) + m(x') + \int m(y) d\mu_0(y), M \Big]_t \lesssim  \lambda_t |\eta_t|^{2p} dt
\end{align*}
since for all $y \in K$, $d \Big[ m(y) , M \Big]_t$ can be written
\begin{align*}
& \left( 2\theta |X_t^\mu(y)|^{2 \theta - 2} X_t^\mu(y) \cdot C[\mu_t](X^\mu_t(y)) \right) 
\left(   
2p |\eta_t|^{2p-2} \eta_t \cdot \left( C[\mu_t](X^\mu_t(x)) - C[\mu_t](X^\mu_t(x')) \right) 
\right) dt
\\
& \lesssim | X_t^\mu(y)|^{2 \theta-1} \left( 1 + |X_t^\mu(x)|^\theta + |X_t^\mu(x')|^\theta + \int |z|^\theta d\mu_t(z)  \right) |\eta_t|^{2p} dt
\\
& \lesssim \left( 1 + |X_t^\mu(x)|^{2 \theta + \theta-1} +  |X_t^\mu(x')|^{2 \theta + \theta-1} 
+  |X_t^\mu(y)|^{2 \theta + \theta-1}
+ \int |z|^{2\theta + \theta-1} d\mu_t(z)  
 \right) |\eta_t|^{2p}dt
 \\
 & \lesssim \left( 1 + |X_t^\mu(x)|^{2 \theta} +  |X_t^\mu(x')|^{2 \theta} 
+  |X_t^\mu(y)|^{2 \theta}
+ \int |z|^{2\theta} d\mu_t(z)  
 \right) |\eta_t|^{2p}dt .
\end{align*}
Note that $\lambda_t \ge 1$ so that $\lambda_t \le \lambda_t^2$. The proof is complete.
\end{proof}

\end{subsection}

\begin{subsection}{Proof of the strong convergence.} \label{chap2-strong}

Now that the support estimate from Corollary \ref{coro} is acquired, we can prove Theorem \ref{chap2-thm2}.
\vspace{3mm}

Let us fix some $T > 0$ and $p \ge 2$.  
Let $K \subset \mathbb{R}^d$ be a compact set, $\mu_0 \in {\cal P}(\mathbb{R}^d)$ be a probability measure with support in $K$, and $(\mu_0^N = \frac{1}{N} \sum_{i=1}^N \delta_{x_0^{i,N}})_N$ be a sequence of empirical measures with support in $K$ such that
\begin{align*}
W_p[\mu_0^N, \mu_0] \xrightarrow[N \to \infty]{} 0.
\end{align*}
Letting $(X_t^{i,N})_{t \in [0,T]}$ be the solution of \myref{model} with initial data $X_0^{i,N} = x_0^{i,N}$, let us introduce $\mu^N = \frac{1}{N} \sum_{i=1}^N \delta_{X^{i,N}}$ the empirical measure associated with the particle system, which is naturally of the transport form: $\mu^N = (X^{\mu^N})^*\mu_0^N$. ~~\\
Consequently,
for $N,M \ge 1$, Proposition \ref{comparison} gives, for all $R > 0$,
\begin{align*}
\E \left[ \sup_{t \in [0, \tau^{N,M}_R]} W_p^p[\mu_t^N, \mu_t^{M}] \right] \le C_{p,R,T} W_p^p[\mu_0^N, \mu_0^M ]
\end{align*}
where the stopping time is given by
\begin{align}
\tau_R^{N,M} = \inf \left\{ t \ge 0, \; \sup_{x \in K} \Big( |X^{\mu^N}_t(x)| + |X^{\mu^M}_t(x)| \Big) \ge R  \right\} \wedge T .
\end{align}
Let us in fact be more precise and use the inequality \myref{ineq_J}: introducing an optimal plan $\pi^{N,M} \in \Pi(\mu_0^N, \mu_0^M)$ so that
\begin{align*}
W_p^p[\mu^N_0, \mu^M_0] = \int_{K^2} |x-y|^p d\pi^{N,M}(x,y),
\end{align*}
we have
\begin{align}
W_p^p[\mu^N, \mu^M] \le J^*_T:= \int_{K^2}  \sup_{t \in [0,T]} \Big| X_t^{\mu^N}(x) - X_t^{\mu^M}(y) \Big|^p d\pi^{N,M}(x,y),
\end{align}
with the inequality
\begin{align}
\E \left[ J^*_{\tau_R^{N,M}} \right] \le C_{p,R,T} W_p^p[\mu_0^N, \mu_0^M].
\end{align}
It then follows that
\begin{align}
\E \left[ W_p^p[\mu^N, \mu^{M}] \right]  & \le \E \left[ J^*_T \right]  \le     \label{trick}
\E \left[ J^*_T  \mathbb{1}_{\left\{\tau_R^{N,M} = T \right\}} \right] 
+
\E \left[ J^*_T \mathbb{1}_{ \left\{ \tau_R^{N,M} < T \right\} } \right] 
\\ \nonumber
& \le \E \left[ J^*_{\tau^{N,M}_R} \right] 
+ C_p \E \left[ \mathbb{1}_{ \left\{ \tau^{N,M}_R < T \right\} } 
\Big( \int_{x \in {\cal C}} \| x \|_\infty^p d\mu^N(x) + \int_{x \in {\cal C}} \| x \|_\infty^p d\mu^M(x) \Big) \Big]
 \right]
\\ \nonumber
& \le C_{p,R,T} W_p^p[\mu_0^N, \mu_0^M] + C_p \Proba\left( \tau^{N,M}_R < T \right)^{1/2}
 \left( \sup_N \E \left[ \int_{x \in {\cal C}} \| x \|^{2p}_\infty d\mu^N(x) \right] \right)^{1/2} .
\end{align}
Therefore, using the bound from Proposition \ref{chap2-moments},
\begin{align*}
\E \left[ W_p^p[\mu^N, \mu^{M}] \right]
  \le C_{p,R,T} W_p^p[\mu_0^N, \mu_0^M]  + C_{p,K} \Proba\left( \tau^{N,M}_R < T \right)^{1/2}  .
\end{align*}
Now using Corollary \ref{coro}, Markov's inequality leads to
\begin{align*}
\Proba\left( \tau^{N,M}_R < T \right) \le \Proba \left( \sup_{t \in [0,T]} \sup_{x \in K} \left( |X_t^{\mu^N}(x)| + |X^{\mu^M}_t(x)| \right) \ge R \right) \le C_{p,K} R^{-2} .
\end{align*}
We conclude that
\begin{align*}
\limsup_{N,M \to \infty}  \E \left[ W_p^p[\mu^N, \mu^{M}] \right] \le C_{p,K} R^{-1}
\xrightarrow[R \to 0]{} 0.
\end{align*}
This shows that $(\mu^N)_{N \ge 1}$ is a Cauchy sequence, hence converges in $L^p(\Omega ; {\cal P}_p({\cal C}))$ to some $\mu$. 
\vspace{3mm}

Let us now prove that the limiting measure $\mu$ is also of the transport form and therefore satisfies the SPDE \myref{chap2-spde} with initial data $\mu_0$ (according to Proposition \ref{transport_solution}). ~~\\
Denoting $\nu = (X^\mu)^* \mu_0 \in {\cal P}_p({\cal C})$, 
we may slightly adapt the proof of Proposition \ref{comparison} to obtain an estimate of $W_p^p[\mu^N,\nu]$. Indeed, introducing an optimal plan $\pi^N \in \Pi(\mu_0^N, \mu_0)$ so that
\begin{align*}
W_p^p[\mu_0^N, \mu_0] = \int_{K^2} |x-y|^p d\pi^N(x,y),
\end{align*}
we have this time
\begin{align}
W_p^p[\mu^N, \nu] \le I^*_T:= \int_{K^2} \sup_{t \in [0,T]} \Big| X_t^{\mu^N}(x) - X_t^{\mu}(y) \Big|^p d\pi^N(x,y)
\label{def_I}
\end{align}
and we are naturally led to study $\eta_t(x,y) = X_t^{\mu^N}(x) - X_t^{\mu}(y) $. Setting $p=2$ for simplicity, calculations give, as in \myref{bdgg},
\begin{align*}
d |\eta_{t \wedge \tau^N_R}(x,y)|^2 & \le C_R \left( |\eta_{t \wedge \tau^N_R}(x,y)|^2 + W_2^2[\mu^N_{t \wedge \tau_R} , \mu_{t \wedge \tau^N_R}]  \right) dt + dM_{t \wedge \tau^N_R}(x,y)
\\ & \le C_R \left( |\eta_{t \wedge \tau^N_R}(x,y)|^2 + W_2^2[\mu^N , \mu]  \right) dt + dM_{t \wedge \tau^N_R}(x,y) 
\end{align*}
where the stopping time is defined as
\begin{align*}
\tau_R^{N} = \inf \left\{ t \ge 0, \; \sup_{x \in K} \Big( |X^{\mu^N}_t(x)| + |X^{\mu}_t(x)| \Big) \ge R  \right\} \wedge T .
\end{align*}
We can now carry on as in the proof of Proposition \ref{comparison} to obtain the estimate
\begin{align}
\E \left[I^*_{\tau^N_R} \right] \le C_{p,R,T} \left( W_p^p[\mu_0^N, \mu_0] +\E \left[ W_p^p[\mu^N, \mu] 
\right]  \right)
\le C'_{p,R,T} \E \left[ W_p^p[\mu^N, \mu] \right]  .
\label{thisJ}
\end{align}
From \myref{def_I} and \myref{thisJ}, using the same method as in \myref{trick}, since $\mu_0^N, \mu_0$ are supported in $K \subset \mathbb{R}^d$, we are led to
\begin{align*}
\E \left[ W_p^p[\mu^N, \nu] \right] \le C_{p,R,T}  \E \left[ W_p^p[\mu^N, \mu] \right] + 
C_{p,K} \Proba \left( \tau^N_R < T \right)^{1/2} .
\end{align*}
Sending $N, R \to \infty$ in the same fashion as before, we conclude that $\E \left[ W_p^p[ \mu, \nu ] \right] = 0$, that is $\mu = \nu$ almost surely: therefore $\mu$ is of the transport form.
\vspace{3mm}

Lastly, the uniqueness of a solution of the transport form for $\mu_0$ supported in $K \subset \mathbb{R}^d$ is again obtained in the same way: given $\mu = (X^\mu)^* \mu_0$ and $\widetilde \mu = (X^{\tilde \mu})^* \mu_0$ two solutions of the transport form, we may apply Proposition \ref{comparison} (or more precisely equation \myref{ineq_J}) and the method used in \myref{trick} to obtain
\begin{align*}
 \E \left[ W_p^p[ \mu, \widetilde \mu] \right] \le C_{p,K} \Proba \left( \tau_R < T \right)^{1/2}
\end{align*} 
with
\begin{align*}
\tau_R = \inf \left\{ t \ge 0, \; \sup_{x \in K} \Big( |X^{\mu}_t(x)| + |X^{\widetilde \mu}_t(x)| \Big) \ge R  \right\} \wedge T .
\end{align*}
We may then send $R \to \infty$ to get $\mu = \widetilde \mu$ almost surely. 
This concludes the proof of Theorem \ref{chap2-thm2}.

\end{subsection}

\begin{subsection}{Propagation of chaos.}
We now prove Theorem \ref{thm3}. To this purpose, we naturally extend the definition of the stochastic characteristics introduced in \myref{chap2-chara} in the following way. ~~\\
Given some random measure $\mu : \Omega \to {\cal P}_2({\cal C})$ with $\E \Big[ \int \| x \|_\infty^2 d\mu(x) \Big] < \infty$ and an ${\cal F}_0$-measurable random variable $\xi_0 : \Omega \to \mathbb{R}^d$, we shall denote by $X^\mu(\xi_0)=(X^\mu_t(\xi_0))_{t \in [0,T]}$ the solution of
\begin{equation}
\left\{
\begin{array}{l}
 dX_t^\mu(\xi_0) = \Big( B[\mu_t] + S[\mu_t] \Big)(X^\mu_t(\xi_0)) dt + C[\mu_t](X^\mu_t(\xi_0)) d\beta_t, \; \; t \in [0,T],
\\
 X_0^\mu(\xi_0) = \xi_0.
 \end{array}
 \right.
\label{chara2}
\end{equation}
Let us start by establishing a link between the law of $X^\mu(\xi_0)$ given by \myref{chara2} and the flow of characteristics defined in \myref{chap2-chara}.

\begin{proposition} \label{conditional_law}
~~\\
Let $\mu : \Omega \to {\cal P}_2({\cal C})$ such that $\E \int \|x\|^2_\infty d\mu(x) < \infty$. ~~\\
Let $\xi_0 : \Omega \to \mathbb{R}^d$ be an ${\cal F}_0$-measurable random variable with law $\mu_0$ supported in some $K \subset \mathbb{R}^d$. 
~~\\
Then, letting $X^\mu(\xi_0)$ be defined by \myref{chara2}, for all $\phi \in C_b({\cal C})$,
$$
\E \left[ \phi(X^\mu(\xi_0)) \Big| {\cal F}^\beta_T \right] = \int_{\mathbb{R}^d} \phi(X^\mu(x)) d\mu_0(x) \; \; \text{a.s},
$$
where $(X^\mu(x))_{x \in \mathbb{R}^d}$ is defined in \myref{chap2-chara}.
\end{proposition}

\begin{proof}[Proof]
Let us first consider the case where $\xi_0$ takes a finite number of values : let us introduce a partition 
$(A_k)_{k \in \{ 1, \ldots, n \} } \in ({\cal F}_0)^n$ of $\Omega$ and $(x_k)_{k \in \{1 , \ldots, n \}} \in (\mathbb{R}^d)^n$ such that
\begin{align}
\xi_0 = \sum_{k=1}^n x_k \mathbb{1}_{A_k}.
\label{finite}
\end{align}
One can easily prove that, in this case,
$$
X^\mu_t(\xi_0) = \sum_{k=1}^n X^\mu_t(x_k) \mathbb{1}_{A_k}, \; \; t \in [0,T], \; \; \text{a.s}
$$
by checking that the left and right hand side both satisfy problem \myref{chara2}, for which path-wise uniqueness is established. Given $\phi \in C_b({\cal C})$, it follows that
\begin{align*}
\E \left[ \phi( X^\mu(\xi_0)) \Big| {\cal F}_T^\beta \right] & = 
\sum_{k=1}^n \E \left[ \phi( X^\mu(x_k)) \mathbb{1}_{A_k} \Big| {\cal F}_T^\beta \right]
\\
& = \sum_{k=1}^n \phi(X^\mu(x_k)) \mu_0(\{ x_k \})
= \int_{\mathbb{R}^d} \phi(X^\mu(x)) d\mu_0(x),
\end{align*}
where we have used the facts that $X^\mu(x) \in {\cal C}$ is ${\cal F}^\beta_T$-measurable and ${\cal F}_0$ is independent of ${\cal F}^\beta_T$.
\\
Given a general random variable $\xi_0$ of law $\mu_0$, let us introduce a sequence $(\xi_0^k)_{k \ge 1}$ of random variables of the form \myref{finite} such that $\xi_0^k \to \xi_0$ a.s. Denoting by $\mu_0^k$ the law of $\xi_0^k$, since the almost sure convergence implies the convergence of the laws, we deduce that $\mu_0^k \to \mu_0$ weakly. For all $k \ge 1$,
$$
\E \left[ \phi(X^\mu(\xi_0^k)) \Big| {\cal F}^\beta_T \right] = \int_{\mathbb{R}^d} \phi(X^\mu(x)) d\mu_0^k(x) \; \; \text{a.s}.
$$
Recalling that the mapping $x \mapsto X^\mu(x) \in {\cal C}$ is almost-surely continuous (see Remark \ref{chap2-continuity}), the result is deduced by taking the limit in $L^1(\Omega)$ as $k$ goes to infinity, using the dominated convergence theorem, 
\end{proof}

With definition \myref{chara2} in mind, introducing the unique solution $\mu$ of \myref{conservation} of the transport form (given by Theorem \ref{chap2-thm2}) the variables considered in Theorem \ref{thm3} may be rewritten as
\begin{align}
X^{i,N} = X^{\mu^N}(\xi_0^i), \; \; \; \; X^i = X^{\mu}(\xi_0^i),
\end{align}
where the empirical measure is of the transport form $\mu^N = (X^{\mu^N})^*\mu_0^N$ with the random initial measure $\mu_0^N = \frac{1}{N} \sum_{i=1}^N \delta_{\xi_0^i}$. Note that the strong law of large numbers gives, for any $\psi \in C_b(\mathbb{R}^d)$,
$$
\langle \psi, \mu_0^N \rangle = \frac{1}{N} \sum_{i=1}^N \psi(\xi_0^i) \xrightarrow[N \to \infty]{} \int \psi(x) d\mu_0(x) \; \; \text{a.s}.
$$
Since $\mu_0^N$ and $\mu_0$ are supported in the compact set $K \subset \mathbb{R}^d$, we easily deduce that
$$
\mu_0^N \to \mu_0 \text{ in } L^p(\Omega ; {\cal P}_p(\mathbb{R}^d)).
$$
From here, we may easily proceed as in the proof of Proposition \ref{comparison} to obtain
$$
\E \left[ \sup_{t \in [0, \tau_R^N]} W_p^p[\mu^N_t, \mu_t] \right] \le C_{p,R,T} \E \left[ W_p^p[\mu_0^N, \mu_0] \right] \xrightarrow[N \to \infty]{} 0,
$$
where the stopping time is defined as
$$
\tau_R^{N} = \inf \left\{ t \ge 0, \; \sup_{x \in K} \Big( |X^{\mu^N}_t(x)| + |X^{\mu}_t(x)| \Big) \ge R  \right\} \wedge T .
$$
The arguments developed in section \ref{chap2-strong} then provide the first convergence in Theorem \ref{thm3}:
$$
\E \left[W_p^p[\mu^N, \mu]  \right] \to 0.
$$
The mean-field limit being established, the propagation of chaos announced in \myref{chaos} will be deduced by exploiting the following symmetry property: using (a slightly tweaked version of) Proposition~\ref{conditional_law}, one can write, for all $r \ge 1$ and $\phi \in C_b({\cal C}^r)$,
\begin{align*}
\E \left[ \phi(X^{1,N}, \ldots, X^{r,N} ) \Big| {\cal F}^\beta_T \right]
& = \E \left[ \phi(X^{\mu^N}(\xi_0^1), \ldots, X^{\mu^N}(\xi_0^r) ) \Big| {\cal F}^\beta_T \right]
\\
& = \int_{(\mathbb{R}^d)^r}
\phi(X^{\mu^N}(x_1), \ldots, X^{\mu^N}(x_r) ) d\mu_0^{\otimes r}(x_1, ... , x_r) \; \; \text{a.s}
\end{align*}
since, by independence, ${\cal L}(\xi_0^1, \ldots, \xi_0^r) = \mu_0^{\otimes r}$. In particular, this shows that
\begin{align}
\forall \sigma \in {\cal S}_r, \; \; \; 
\E \left[ \phi(X^{\sigma(1),N}, \ldots, X^{\sigma(r),N} ) \Big| {\cal F}^\beta_T \right] = \E \left[ \phi(X^{1,N}, \ldots, X^{r,N} ) \Big| {\cal F}^\beta_T \right] \; \; \text{a.s}.
\label{symmetry}
\end{align}
The convergence \myref{chaos} can now be proved in the same way as in \cite{flandoli}, Theorem 24. We detail the proof here for the sake of completeness: for simplicity, let us consider the case $r=2$. For $\phi_1, \phi_2 \in C_b({\cal C})$ with $|\phi_1|, |\phi_2| \le M$, let us write
\begin{align*}
\E \Big| \E \left[ \phi_1(X^{1,N}) \phi_2(X^{2,N}) \Big| {\cal F}^\beta_T \right] - \langle \phi_1, \mu \rangle \langle \phi_2 , \mu \rangle \Big| \le \E [A^N] + \E [B^N]
\end{align*}
where
\begin{align*}
& A^N = \Big| \E \left[ \phi_1(X^{1,N}) \phi_2(X^{2,N}) \Big| {\cal F}^\beta_T \right] - 
\E \left[ \langle \phi_1, \mu^N \rangle \langle \phi_2 , \mu^N \rangle \Big| {\cal F}^\beta_T \right] \Big|,
\\
& B^N = \Big| \E \left[ \langle \phi_1, \mu^N \rangle \langle \phi_2 , \mu^N \rangle \Big| {\cal F}^\beta_T \right]
- \E \left[ \langle \phi_1, \mu \rangle \langle \phi_2 , \mu \rangle \Big| {\cal F}^\beta_T \right] \Big|.
\end{align*}
On one hand, using the symmetry property \myref{symmetry}, we may rewrite $A^N$ as
\begin{align*}
A^N & = \Big|
\frac{1}{N^2-N} \sum_{i \neq j} \E \left[ \phi_1(X^{i,N}) \phi_2(X^{j,N}) \Big| {\cal F}^\beta_T \right]
- \frac{1}{N^2} \sum_{i,j}  \E \left[ \phi_1(X^{i,N}) \phi_2(X^{j,N}) \Big| {\cal F}^\beta_T \right]
\Big|
\\
& \le \left( \frac{1}{N^2-N} - \frac{1}{N^2} \right)(N^2-N)M^2 + \frac{1}{N}M^2 = 2 \frac{M^2}{N} \to 0 .
\end{align*}
On the other hand, 
\begin{align*}
\E [ B^N ] \le \E \Big| \langle \phi_1, \mu^N \rangle \langle \phi_2, \mu^N \rangle - \langle \phi_1, \mu \rangle \langle \phi_2, \mu \rangle \Big| \to 0
\end{align*}
using the dominated convergence theorem, given that $\mu^N \to \mu$ in ${\cal P}_p({\cal C})$ (hence in particular $\langle \phi , \mu^N \rangle \to \langle \phi, \mu \rangle$) in probability. This proves \myref{chaos}.
~~\\
Furthermore, for all $\phi \in C_b({\cal C})$, since $\mu = (X^\mu)^* \mu_0$ is of the transport form, Proposition \ref{conditional_law} gives
$$
\E \left[ \phi(X^i) \Big| {\cal F}^\beta_T \right]  =  
\E \left[ \phi(X^\mu(\xi_0^i)) \Big| {\cal F}^\beta_T \right] = \int_{\mathbb{R}^d} \phi(X^\mu(x)) d\mu_0(x)
= \int_{\cal C} \phi(x) d\mu(x)
 \; \; \text{a.s}
$$
so that $\mu \in {\cal P}({\cal C})$ is indeed a version of the conditional law ${\cal L} \left( X^i \Big| {\cal F}^\beta_T \right)$.
~~\\
 Lastly, one may again use (a slightly tweaked version of) Proposition \ref{conditional_law} to write
\begin{align*}
\E \| X^{i,N} - X^i \|_\infty^p = \E \| X^{\mu^N}(\xi_0^i) - X^\mu(\xi_0^i) \|_\infty^p
= \E \int_{K} \sup_{t \in [0,T]} \Big| X_t^{\mu^N}(x) - X_t^\mu(x) \Big|^p d\mu_0(x).
\end{align*}
Once again, the arguments developed in section \ref{chap2-strong} (see \myref{def_I} and below) give the required convergence: we obtain $X^{i,N} \to X^i$ in $L^p(\Omega ; {\cal C})$, which concludes the proof of Theorem~\ref{thm3}.

\end{subsection}

\begin{subsection}*{Acknowledgment}

A. Rosello is partially supported by the French government thanks to the "Investissements d'Avenir"
program ANR-11-LABX-0020-01. 

\end{subsection}

\end{section}

\nocite{*}
\bibliographystyle{plain}
\bibliography{particle}


\end{document}